\documentclass[a4paper,11pt]{article}
\usepackage{amsmath,amssymb,amsthm}
\usepackage{amscd}

\makeatletter

\@addtoreset{equation}{subsection}
\makeatother

\newtheorem{dfn}{Definition}[subsection]
\newtheorem{prop}[dfn]{Proposition}
\newtheorem{thm}[dfn]{Theorem}
\newtheorem{lem}[dfn]{Lemma}
\newtheorem{cor}[dfn]{Corollary}

\newtheorem{rem}[dfn]{Remark}
\newtheorem{quest}[dfn]{Question}

\begin{document}

\title
{The generalized Dehn twist along a figure eight}
\author{Yusuke Kuno}
\date{}
\maketitle

\begin{abstract}
For any unoriented loop on a compact connected oriented
surface with one boundary component, the generalized Dehn
twist along the loop is defined as an automorphism of
the completed group ring of the fundamental group of the surface. If the loop is
simple, this is the usual right handed Dehn twist, in particular
realized as a mapping class of the surface. We investigate the case
when the loop has a single transverse double point, and show that
in this case the generalized Dehn twist is not realized as a mapping class.
\end{abstract}

\section{Introduction}
The right handed Dehn twist $t_C$ along a simple closed curve $C$
on an oriented surface is a diffeomorphism of the surface, which
is illustrated in Figure 1. By definition, a Dehn twist is local;
the support of a Dehn twist lies in a regular neighborhood of the curve.

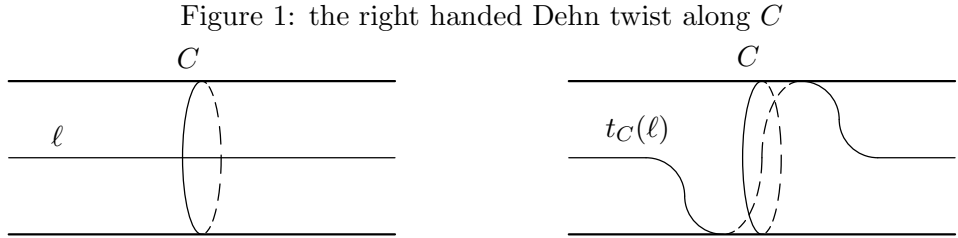
\begin{center}
Figure 1: the right handed Dehn twist along $C$

\unitlength 0.1in
\begin{picture}( 49.0000, 10.5000)(  2.0000,-12.0000)
%
\special{pn 13}%
\special{pa 200 400}%
\special{pa 2200 400}%
\special{fp}%
%
\special{pn 13}%
\special{pa 200 1200}%
\special{pa 2200 1200}%
\special{fp}%
%
\special{pn 13}%
\special{pa 3100 400}%
\special{pa 5100 400}%
\special{fp}%
%
\special{pn 13}%
\special{pa 3100 1200}%
\special{pa 5100 1200}%
\special{fp}%
%
\special{pn 8}%
\special{ar 1200 800 100 400  1.5707963 4.7123890}%
%
\special{pn 8}%
\special{ar 1200 800 100 400  4.7123890 4.9523890}%
\special{ar 1200 800 100 400  5.0963890 5.3363890}%
\special{ar 1200 800 100 400  5.4803890 5.7203890}%
\special{ar 1200 800 100 400  5.8643890 6.1043890}%
\special{ar 1200 800 100 400  6.2483890 6.4883890}%
\special{ar 1200 800 100 400  6.6323890 6.8723890}%
\special{ar 1200 800 100 400  7.0163890 7.2563890}%
\special{ar 1200 800 100 400  7.4003890 7.6403890}%
\special{ar 1200 800 100 400  7.7843890 7.8539816}%
%
\special{pn 8}%
\special{ar 4100 800 100 400  1.5707963 4.7123890}%
%
\special{pn 8}%
\special{ar 4100 800 100 400  4.7123890 4.9523890}%
\special{ar 4100 800 100 400  5.0963890 5.3363890}%
\special{ar 4100 800 100 400  5.4803890 5.7203890}%
\special{ar 4100 800 100 400  5.8643890 6.1043890}%
\special{ar 4100 800 100 400  6.2483890 6.4883890}%
\special{ar 4100 800 100 400  6.6323890 6.8723890}%
\special{ar 4100 800 100 400  7.0163890 7.2563890}%
\special{ar 4100 800 100 400  7.4003890 7.6403890}%
\special{ar 4100 800 100 400  7.7843890 7.8539816}%
%
\special{pn 8}%
\special{pa 200 800}%
\special{pa 2200 800}%
\special{fp}%
%
\special{pn 8}%
\special{ar 3900 800 200 400  6.2831853 6.4831853}%
\special{ar 3900 800 200 400  6.6031853 6.8031853}%
\special{ar 3900 800 200 400  6.9231853 7.1231853}%
\special{ar 3900 800 200 400  7.2431853 7.4431853}%
\special{ar 3900 800 200 400  7.5631853 7.7631853}%
%
\special{pn 8}%
\special{ar 4300 800 200 400  3.1415927 3.3415927}%
\special{ar 4300 800 200 400  3.4615927 3.6615927}%
\special{ar 4300 800 200 400  3.7815927 3.9815927}%
\special{ar 4300 800 200 400  4.1015927 4.3015927}%
\special{ar 4300 800 200 400  4.4215927 4.6215927}%
%
\special{pn 8}%
\special{ar 3900 1000 200 200  1.5707963 3.1415927}%
%
\special{pn 8}%
\special{ar 3500 1000 200 200  4.7123890 6.2831853}%
%
\special{pn 8}%
\special{pa 3100 800}%
\special{pa 3500 800}%
\special{fp}%
%
\special{pn 8}%
\special{ar 4300 600 200 200  4.7123890 6.2831853}%
%
\special{pn 8}%
\special{ar 4700 600 200 200  1.5707963 3.1415927}%
%
\special{pn 8}%
\special{pa 5100 800}%
\special{pa 4700 800}%
\special{fp}%
\put(4.2000,-7.4000){\makebox(0,0)[lb]{$\ell$}}%
\put(10.7000,-3.3000){\makebox(0,0)[lb]{$C$}}%
\put(39.7000,-3.2000){\makebox(0,0)[lb]{$C$}}%
\put(32.9000,-7.3000){\makebox(0,0)[lb]{$t_C(\ell)$}}%
\end{picture}%
\end{center}

The purpose of this paper is to give a generalization of Dehn twists
for not necessarily simple loops and begin the study of it.
We will consider on a compact connected oriented surface $\Sigma$ with
one boundary component. In this case the mapping class group $\mathcal{M}_{g,1}$
of the surface faithfully acts on the fundamental group $\pi$ of the surface.
The results of this paper are summarized as follows:
\begin{enumerate}
\item For each unoriented loop $\gamma \subset \Sigma$,
the generalized Dehn twist along $\gamma$, denoted by $t_{\gamma}$,
is defined as an automorphism of the completed group ring $\widehat{\mathbb{Q}\pi}$
(Definition \ref{GDT}).
\item When $\gamma$ is simple, $t_{\gamma}$ is the usual right handed Dehn twist
along $\gamma$.
\item We consider whether $t_{\gamma}$ is a mapping class, i.e.,
$t_{\gamma}$ lies in the image of the natural injection
$\mathcal{M}_{g,1} \hookrightarrow {\rm Aut}(\widehat{\mathbb{Q}\pi})$.
We show that if $t_{\gamma}$ is a mapping class, then it must be local;
there is a diffeomorphism representing $t_{\gamma}$,
whose support lies in a regular neighborhood of $\gamma$
(Theorem \ref{support}).
\item Using the criterion above, we show that when $\gamma$ has
a single transverse double point and is not homotopic to
a power of a simple closed curve (we shall call such a loop
a {\it figure eight}), then $t_{\gamma}$ is
not a mapping class (Theorem \ref{mainthm}).
\end{enumerate}

Our generalization of Dehn twists is based on the results
in \cite{KK}. In that paper, the action of Dehn twists
on the completed group ring of the fundamental group of
the surface is given in terms of an invariant of unoriented
loops on the surface $\Sigma$. This invariant is regarded
as a derivation on the completed group ring $\widehat{\mathbb{Q}\pi}$.
When the loop is simple, this invariant turns out to be
``the logarithms of Dehn twists". Even if the loop is not
simple, the exponential of this invariant still have
its meaning as an automorphism of $\widehat{\mathbb{Q}\pi}$.
This is our generalization.

The proof of the main theorem presented here is rather
ad hoc, since we depend on the classification of the possible
configurations of a figure eight on $\Sigma$ (Proposition \ref{PinSig}),
and explicit computations of tensors for each configuration.
But still, we meet an interesting phenomenon. Namely,
if we assume the generalized Dehn twist along a figure eight $\gamma$
is a mapping class, then it should be written as
$t_{\gamma}=t_{C_1}^2t_{C_2}^2t_{C_3}^{-1}$, where
$C_i$ are the suitable numbered boundary components
of a closed regular neighborhood of $\gamma$ (Proposition \ref{22-1}).
The main theorem is proved by looking
at this equation in higher degree.

\vspace{0.3cm}

\noindent \textbf{Acknowledgments.} I am grateful
to Nariya Kawazumi, for valuable discussions and communicating
to me Proposition \ref{compatibleU}. A part of this work has
been done during my stay at QGM, Aarhus University
from September 2010 to January 2011.
I would like to thank QGM for kind hospitality.
This research is supported by JSPS Research Fellowships for
Young Scientists (22$\cdot$4810).

\tableofcontents

\section{Preliminaries}

In this section we prepare what we need to define generalized Dehn twists.
We first recall materials from \cite{KK}, such as symplectic expansion,
total Johnson map, Kontsevich's ``associative", and the Goldman Lie algebra.
After that we study symplectic derivations on the completed tensor algebra
and algebra automorphisms of the completed tensor algebra preserving
the symplectic form. We end this section by showing that
the mapping class group acts on Kontsevich's ``associative"
through a symplectic expansion.

All the loops, the homotopies, and the isotopies
that we consider are piecewise differentiable.
As usual, we often ignore the distinction between a path
and its homotopy class.

\subsection{Symplectic expansion and total Johnson map}
Let $\Sigma$ be a compact connected oriented $C^{\infty}$-surface
of genus $g>0$ with one boundary component. Taking a basepoint
$*$ on the boundary $\partial \Sigma$ we denote by $\pi:=\pi_1(\Sigma,*)$
the fundamental group of $\Sigma$, which is free of rank $2g$.
Let $\zeta\in \pi$ be a based loop
parallel to $\partial \Sigma$ and going by counter-clockwise
manner. If we take symplectic generators $\alpha_1,\beta_1,
\ldots, \alpha_g,\beta_g \in \pi$ as shown in Figure 2, then
$\zeta=\prod_{i=1}^g [\alpha_i,\beta_i]$. Here, for based
loops $x$ and $y$, their product $xy$ means that $x$ is traversed
first, and $[\alpha_i,\beta_i]=\alpha_i\beta_i\alpha_i^{-1}\beta_i^{-1}$.

\begin{center}
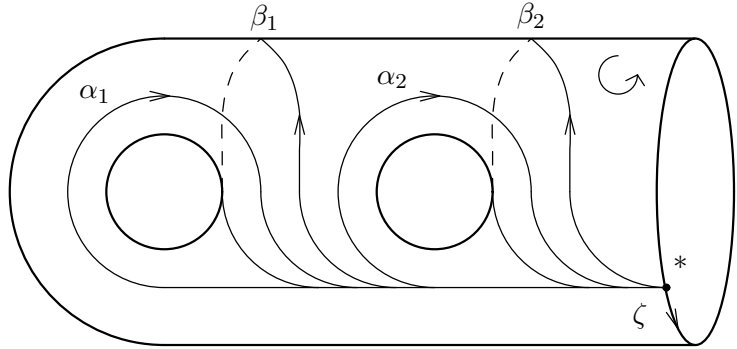
Figure 2: symplectic generators for $g=2$

\unitlength 0.1in
\begin{picture}( 37.5000, 18.2000)(  2.0000,-18.7000)
%
\special{pn 13}%
\special{ar 3750 1070 200 800  0.0000000 6.2831853}%
%
\special{pn 13}%
\special{ar 1000 1070 800 800  1.5707963 4.7123890}%
%
\special{pn 13}%
\special{ar 1000 1070 300 300  0.0000000 6.2831853}%
%
\special{pn 13}%
\special{ar 2400 1070 300 300  0.0000000 6.2831853}%
%
\special{pn 8}%
\special{ar 1000 1070 500 500  1.5707963 6.2831853}%
%
\special{pn 8}%
\special{ar 2400 1070 500 500  1.5707963 6.2831853}%
%
\special{pn 8}%
\special{ar 2000 1070 500 500  1.5707963 3.1415927}%
%
\special{pn 8}%
\special{ar 3400 1070 500 500  1.5707963 3.1415927}%
%
\special{pn 8}%
\special{ar 1800 1070 500 500  1.5707963 3.1415927}%
%
\special{pn 8}%
\special{ar 3200 1070 500 500  1.5707963 3.1415927}%
%
\special{pn 8}%
\special{ar 2200 1070 500 500  1.5707963 3.1415927}%
%
\special{pn 8}%
\special{ar 3600 1070 500 500  1.5707963 3.1415927}%
%
\special{pn 8}%
\special{pa 2700 1070}%
\special{pa 2700 1038}%
\special{pa 2700 1006}%
\special{pa 2700 974}%
\special{pa 2700 942}%
\special{pa 2700 910}%
\special{pa 2700 878}%
\special{pa 2700 846}%
\special{pa 2700 814}%
\special{pa 2698 782}%
\special{pa 2698 750}%
\special{pa 2698 718}%
\special{pa 2700 686}%
\special{pa 2702 654}%
\special{pa 2706 622}%
\special{pa 2710 590}%
\special{pa 2716 558}%
\special{pa 2726 526}%
\special{pa 2736 494}%
\special{pa 2746 464}%
\special{pa 2760 434}%
\special{pa 2776 406}%
\special{pa 2794 380}%
\special{pa 2812 356}%
\special{pa 2834 332}%
\special{pa 2858 310}%
\special{pa 2882 288}%
\special{pa 2900 270}%
\special{sp 0.070}%
%
\special{pn 8}%
\special{pa 3100 1070}%
\special{pa 3100 1038}%
\special{pa 3100 1006}%
\special{pa 3100 974}%
\special{pa 3100 942}%
\special{pa 3100 910}%
\special{pa 3100 878}%
\special{pa 3100 846}%
\special{pa 3102 814}%
\special{pa 3102 782}%
\special{pa 3102 750}%
\special{pa 3102 718}%
\special{pa 3102 686}%
\special{pa 3100 654}%
\special{pa 3096 622}%
\special{pa 3090 590}%
\special{pa 3084 558}%
\special{pa 3076 526}%
\special{pa 3066 494}%
\special{pa 3054 464}%
\special{pa 3040 434}%
\special{pa 3026 406}%
\special{pa 3008 380}%
\special{pa 2988 356}%
\special{pa 2966 332}%
\special{pa 2944 310}%
\special{pa 2920 288}%
\special{pa 2900 270}%
\special{sp}%
%
\special{pn 8}%
\special{pa 1300 1070}%
\special{pa 1300 1038}%
\special{pa 1300 1006}%
\special{pa 1300 974}%
\special{pa 1300 942}%
\special{pa 1300 910}%
\special{pa 1300 878}%
\special{pa 1300 846}%
\special{pa 1300 814}%
\special{pa 1298 782}%
\special{pa 1298 750}%
\special{pa 1298 718}%
\special{pa 1300 686}%
\special{pa 1302 654}%
\special{pa 1306 622}%
\special{pa 1310 590}%
\special{pa 1316 558}%
\special{pa 1326 526}%
\special{pa 1336 494}%
\special{pa 1346 464}%
\special{pa 1360 434}%
\special{pa 1376 406}%
\special{pa 1394 380}%
\special{pa 1412 356}%
\special{pa 1434 332}%
\special{pa 1458 310}%
\special{pa 1482 288}%
\special{pa 1500 270}%
\special{sp 0.070}%
%
\special{pn 8}%
\special{pa 1700 1070}%
\special{pa 1700 1038}%
\special{pa 1700 1006}%
\special{pa 1700 974}%
\special{pa 1700 942}%
\special{pa 1700 910}%
\special{pa 1700 878}%
\special{pa 1700 846}%
\special{pa 1702 814}%
\special{pa 1702 782}%
\special{pa 1702 750}%
\special{pa 1702 718}%
\special{pa 1702 686}%
\special{pa 1700 654}%
\special{pa 1696 622}%
\special{pa 1690 590}%
\special{pa 1684 558}%
\special{pa 1676 526}%
\special{pa 1666 494}%
\special{pa 1654 464}%
\special{pa 1640 434}%
\special{pa 1626 406}%
\special{pa 1608 380}%
\special{pa 1588 356}%
\special{pa 1566 332}%
\special{pa 1544 310}%
\special{pa 1520 288}%
\special{pa 1500 270}%
\special{sp}%
%
\special{pn 8}%
\special{pa 1000 1570}%
\special{pa 3600 1570}%
\special{fp}%
%
\special{pn 13}%
\special{pa 1000 1870}%
\special{pa 3746 1870}%
\special{fp}%
%
\special{pn 13}%
\special{pa 1000 270}%
\special{pa 3746 270}%
\special{fp}%
%
\special{pn 20}%
\special{sh 1}%
\special{ar 3600 1570 10 10 0  6.28318530717959E+0000}%
\special{sh 1}%
\special{ar 3600 1570 10 10 0  6.28318530717959E+0000}%
%
\special{pn 8}%
\special{pa 1676 770}%
\special{pa 1700 670}%
\special{fp}%
\special{pa 1700 670}%
\special{pa 1726 770}%
\special{fp}%
%
\special{pn 8}%
\special{pa 3076 770}%
\special{pa 3100 670}%
\special{fp}%
\special{pa 3100 670}%
\special{pa 3126 770}%
\special{fp}%
%
\special{pn 8}%
\special{pa 926 546}%
\special{pa 1026 570}%
\special{fp}%
\special{pa 1026 570}%
\special{pa 926 596}%
\special{fp}%
%
\special{pn 8}%
\special{pa 2326 546}%
\special{pa 2426 570}%
\special{fp}%
\special{pa 2426 570}%
\special{pa 2326 596}%
\special{fp}%
%
\special{pn 8}%
\special{pa 3600 1700}%
\special{pa 3660 1784}%
\special{fp}%
\special{pa 3660 1784}%
\special{pa 3646 1682}%
\special{fp}%
\put(5.6000,-6.0500){\makebox(0,0)[lb]{$\alpha_1$}}%
\put(14.4500,-2.3000){\makebox(0,0)[lb]{$\beta_1$}}%
\put(21.0000,-5.3000){\makebox(0,0)[lb]{$\alpha_2$}}%
\put(28.2500,-2.2000){\makebox(0,0)[lb]{$\beta_2$}}%
\put(36.3500,-14.7500){\makebox(0,0)[lb]{$*$}}%
\put(34.2500,-17.8500){\makebox(0,0)[lb]{$\zeta$}}%
%
\special{pn 8}%
\special{ar 3350 460 100 100  6.2831853 6.2831853}%
\special{ar 3350 460 100 100  0.0000000 4.7123890}%
%
\special{pn 8}%
\special{pa 3450 460}%
\special{pa 3390 490}%
\special{fp}%
\special{pa 3450 460}%
\special{pa 3486 516}%
\special{fp}%
\end{picture}%
\end{center}

Let $H:=H_1(\Sigma;\mathbb{Q})$ be the first homology group of
$\Sigma$, and $\widehat{T}:=\prod_{m=0}^{\infty}H^{\otimes m}$
the completed tensor algebra generated by $H$. The algebra $\widehat{T}$
has a decreasing filtration given by
$\widehat{T}_p:=\prod_{m\ge p}^{\infty}H^{\otimes m}$, $p\ge 1$,
and is a complete Hopf algebra with respect to the
coproduct given by $\Delta(X)=X\hat{\otimes}1+1\hat{\otimes}X$,
$X\in H$.
We denote by $A_i,B_i\in H$ the homology class
represented by $\alpha_i,\beta_i$, respectively.
Also, for $x\in \pi$ we denote by $[x]\in H$ the corresponding
homology class.
The two tensor $\omega=\sum_{i=1}^g A_iB_i-B_iA_i\in H^{\otimes 2}$
is independent of the choice of symplectic generators,
and called {\it the symplectic form}. Here and throughout this paper
we often omit $\otimes$ to express tensors.

\begin{dfn}[Kawazumi \cite{Kaw1}]
A map $\theta \colon \pi \to \widehat{T}$ is called
a Magnus expansion of $\pi$ if
\begin{enumerate}
\item for any $x\in \pi$, $\theta(x)\equiv 1+[x] \ ({\rm mod\ }\widehat{T}_2)$,
\item for any $x,y\in \pi$, $\theta(xy)=\theta(x)\theta(y)$.
\end{enumerate}
\end{dfn}

Let $\widehat{\mathbb{Q}\pi}$ be the completed group ring
of $\pi$. Namely, $\widehat{\mathbb{Q}\pi}:=
\varprojlim_m \mathbb{Q}\pi/(I\pi)^m$, where $I\pi$
is the augmentation ideal of the group ring $\mathbb{Q}\pi$.
It has a decreasing filtration given by
$\varprojlim_{m\ge p} (I\pi)^p/(I\pi)^m$, $p\ge 1$.
Any Magnus expansion $\theta$ induces an isomorphism
$\theta\colon \widehat{\mathbb{Q}\pi}\stackrel{\cong}{\to}
\widehat{T}$ of complete augmented algebras.
Kawazumi \cite{Kaw1} introduced this notion to study the automorphism
group of a free group or the mapping class group.
Taking the fact that $\pi$ is the fundamental group of a surface
into consideration, we consider the following notion.

\begin{dfn}[Massuyeau \cite{Mas}]
A Magnus expansion $\theta\colon \pi \to \widehat{T}$ of $\pi$
is called a symplectic expansion if
\begin{enumerate}
\item for any $x\in \pi$, $\theta(x)$ is group-like, i.e., $\Delta(\theta(x))
=\theta(x)\hat{\otimes} \theta(x)$,
\item $\theta(\zeta)=\exp(\omega)$.
\end{enumerate}
\label{SYMPEXP}
\end{dfn}

Here $\exp(\omega)=\sum_{m=0}^{\infty}(1/m!)\omega^m$.
Symplectic expansions do exist \cite{Mas}, and
they are infinitely many \cite{KK} Proposition 2.8.1.
For several constructions, see \cite{Kaw2} \cite{Ku} \cite{Mas}.
Any symplectic expansion $\theta$ induces an isomorphism
$\theta\colon \widehat{\mathbb{Q}\pi} \stackrel{\cong}{\to} \widehat{T}$
of complete Hopf algebras. Moreover, the restriction of $\theta$
to a cyclic subgroup generated by $\zeta$ gives an isomorphism
$\widehat{\mathbb{Q}\langle \zeta \rangle}\stackrel{\cong}{\to} \mathbb{Q}[[\omega]]$.
Actually, for our purpose considering Magnus expansions satisfying
the second condition of Definition \ref{SYMPEXP} is sufficient,
due to Proposition \ref{compatibleU}. But still this notion
would play a significant role in study of the mapping class groups,
as we see in the work of Massuyeau \cite{Mas} where he fully used
the group-like condition of a symplectic expansion.

We denote by $\mathcal{M}_{g,1}$ the mapping class group
of $\Sigma$ relative to the boundary, namely the
group of orientation preserving diffeomorphisms of $\Sigma$
which fix the boundary $\partial \Sigma$ pointwise,
modulo isotopies fixing $\partial \Sigma$ pointwise.
By the theorem of Dehn-Nielsen, we can identify
$$\mathcal{M}_{g,1}=\{ \varphi \in {\rm Aut}(\pi);
\varphi(\zeta)=\zeta \},$$
by looking at the natural action of $\mathcal{M}_{g,1}$
on the fundamental group $\pi$.

If we fix a Magnus expansion $\theta$, the associated
is the notion of the total Johnson map.
We denote by ${\rm Aut}(\widehat{T})$ the group of filter-preserving
algebra automorphisms of $\widehat{T}$.
Let $\varphi \in \mathcal{M}_{g,1}$. Then $\varphi$
induces an isomorphism $\varphi \colon \widehat{\mathbb{Q}\pi}
\to \widehat{\mathbb{Q}\pi}$, hence a uniquely determined automorphism
$T^{\theta}(\varphi)\in {\rm Aut}(\widehat{T})$
satisfying $T^{\theta}(\varphi)\circ \theta
=\theta \circ \varphi$.

\begin{dfn}[Kawazumi \cite{Kaw1}]
The automorphism $T^{\theta}(\varphi)\in {\rm Aut}(\widehat{T})$ is called the
total Johnson map of $\varphi$ associated to $\theta$.
\end{dfn}

The group homomorphism $T^{\theta} \colon \mathcal{M}_{g,1}
\to {\rm Aut}(\widehat{T})$ is injective, since the natural map
$\pi \to \widehat{\mathbb{Q}\pi}$ is injective by the classical fact
$\bigcap_{m=1}^{\infty} (I\pi)^m=0$.

\subsection{Kontsevich's ``associative"}
We define a linear map $N: \widehat{T} \to \widehat{T}_1$ by 
$$N(X_1\cdots X_m)=\sum_{i=1}^m X_i\cdots X_m X_1 \cdots X_{i-1},$$
where $m\ge 1$, $X_i\in H$, and $N(1)=0$.
The following lemma will be used frequently.

\begin{lem}[\cite{KK} Lemma 2.6.2 (1)(2)]
\label{propN}
\begin{enumerate}
\item For $u,v\in \widehat{T}$, $N(uv)=N(vu)$.
\item For $u,v,w\in \widehat{T}$, $N([u,v]w)=N(u[v,w])$.
\end{enumerate}

Here $[u,v]=uv-vu$.
\end{lem}

Let us recall Kontsevich's ``associative" \cite{Kon}.
By definition, a derivation on $\widehat{T}$ is a linear map 
$D: \widehat{T} \to \widehat{T}$ 
satisfying the Leibniz rule:
$$D(u_1u_2) = D(u_1)u_2 + u_1D(u_2),$$
for $u_1, u_2 \in \widehat{T}$.
Since $\widehat{T}$ is freely generated by $H$ as a complete algebra, 
any derivation on $\widehat{T}$ is uniquely determined by its values on $H$, 
and the space of derivations of $\widehat{T}$ is identified with 
${\rm Hom}(H, \widehat{T})$. By the Poincar\'e duality, 
$\widehat{T}_1 \cong H\otimes\widehat{T}$ is identified with ${\rm Hom}(H, \widehat{T})$: 
\begin{equation}\label{T1-Hom}
\widehat{T}_1 \cong H\otimes\widehat{T}\overset\cong\to
{\rm Hom}(H,\widehat{T}), \quad
X \otimes u \mapsto (Y \mapsto (Y \cdot X)u).
\end{equation}
Here $(\ \cdot\ )$ is the intersection pairing on $H = H_1(\Sigma;
\mathbb{Q})$.

Let $\mathfrak{a}_g^-={\rm Der}_{\omega}(\widehat{T})$ be the space of derivations 
on $\widehat{T}$ killing the symplectic form $\omega$. An element of $\mathfrak{a}_g^-$
is called a {\it symplectic derivation}.
In view of (\ref{T1-Hom}) any derivation $D$ is written as 
$$D=\sum^g_{i=1}B_i\otimes D(A_i)-A_i\otimes D(B_i)\in H\otimes \widehat{T}. $$
Since $-D(\omega) = \sum^g_{i=1}[B_i,D(A_i)] - [A_i,D(B_i)]$, 
we have $\mathfrak{a}_g^-={\rm Ker}([\ ,\ ]: H \otimes\widehat{T}\to \widehat{T})$. 
It is easy to see ${\rm Ker}([\ ,\ ]) = N(\widehat{T}_1)$ (see \cite{KK} 
Lemma 2.6.2 (4)).
Hence we can write
$$
\mathfrak{a}_g^- = {\rm Ker}([\ ,\ ]: H\otimes\widehat{T} \to \widehat{T}) = N(\widehat{T}_1).
$$
The Lie subalgebra $\mathfrak{a}_g := N(\widehat{T}_2)$ is nothing but (the completion of)
what Kontsevich \cite{Kon} calls $a_g$.

By a straightforward computation, we have

\begin{lem}
\label{Nuvhomo}
Let $m,n\ge 1$ and $X_1,\ldots,X_m,Y_1,\ldots,Y_n \in H$. Then
we have
$$[N(X_1\cdots X_m),N(Y_1\cdots Y_n)]=N(N(X_1\cdots X_m)(Y_1\cdots Y_n)).$$
Here the left hand side means the Lie bracket of the two derivations
in $\mathfrak{a}_g^-=N(\widehat{T}_1)$
and in the right hand side $N(X_1\cdots X_m)(Y_1\cdots Y_n)$ means
the action of $N(X_1\cdots X_m)\in \mathfrak{a}_g^-$ on the tensor
$Y_1\cdots Y_n\in \widehat{T}$ as a derivation.
\end{lem}

\begin{cor}
\label{Nuv}
Let $u\in N(\widehat{T}_1)$ and $v\in \widehat{T}_1$. Then
$$[u,N(v)]=N(u(v)).$$
Here the left hand side means the Lie bracket of
$u,N(v)\in \mathfrak{a}_g^-=N(\widehat{T}_1)$ and
in the right hand side $u(v)$ means the action of $u\in \mathfrak{a}_g^-$ on the
tensor $v\in \widehat{T}$ as a derivation.
\end{cor}

\begin{proof}
If $u$ and $v$ are homogeneous, this is Lemma \ref{Nuvhomo}.
The general case follows from the bi-linearity.
\end{proof}

Let ${\rm IA}(\widehat{T})$ be the subgroup of ${\rm Aut}(\widehat{T})$
consisting of automorphisms which act on $\widehat{T}_1/\widehat{T}_2\cong H$
as the identity. This group is identified with ${\rm Hom}(H,\widehat{T}_2)=
H\otimes \widehat{T}_2$,
by the logarithms:
\begin{equation}
\label{IA(T)}
{\rm IA}(\widehat{T}) \stackrel{\cong}{\to} {\rm Hom}(H,\widehat{T}_2)
= H\otimes \widehat{T}_2,\ U \mapsto (\log U)|_H.
\end{equation}
Let ${\rm IA}_{\omega}(\widehat{T})$ be the subgroup of ${\rm IA}(\widehat{T})$
consisting of automorphisms preserving $\omega$. By the same argument in
\cite{KK} Proposition 2.8.1, we see the bijection (\ref{IA(T)}) gives a bijection
$$
{\rm IA}_{\omega}(\widehat{T}) \stackrel{\cong}{\to}
{\rm Ker}([\ ,\ ]\colon H \otimes \widehat{T}_2
\to \widehat{T})=N(\widehat{T}_3).
$$

The following proposition was communicated to the author by Nariya Kawazumi.
\begin{prop}
\label{compatibleU}
Assume $U\in {\rm Aut}(\widehat{T})$ satisfies $U(\omega)=\omega$,
and let $v\in \widehat{T}_1$. Then
$$U(Nv)U^{-1}=N(Uv).$$
Here $U(Nv)U^{-1}$ means the conjugate of the derivation $Nv$ by
the algebra automorphism $U$.
\end{prop}

\begin{proof}
Let $|U|$ be the element of ${\rm Aut}(\widehat{T})$ induced from
the action of $U$ on $\widehat{T}_1/\widehat{T}_2 \cong H$.
Then $|U|$ preserves the intersection form $\omega$, and
$U=U^{\prime}\circ |U|$, where $U^{\prime} \in {\rm IA}_{\omega}(\widehat{T})$.
Therefore, it suffices to prove the formula for $U=|U|$ and $U\in {\rm IA}_{\omega}(\widehat{T})$.

Suppose $U=|U|$.
We may assume $v$ is homogeneous and $v=X_1 \cdots X_m$,
where $m\ge 1$ and $X_i \in H$. Then for $Y\in H$,
\begin{eqnarray*}
U(Nv)U^{-1} (Y) &=& U (N(X_1\cdots X_m)(U^{-1}Y)) \\
&=& U(\textstyle{\sum_i} (U^{-1}Y\cdot X_i) X_{i+1}\cdots X_m X_1\cdots X_{i-1}) \\
&=& \sum_i (Y\cdot UX_i) U(X_{i+1}\cdots X_m X_1 \cdots X_{i-1}) \\
&=& N(U(X_1 \cdots X_m))(Y)=N(Uv)(Y),
\end{eqnarray*}
hence $U(Nv)U^{-1}=N(Uv)$. Suppose $U\in {\rm IA}_{\omega}(\widehat{T})$.
By Corollary \ref{Nuv}, ${\rm ad}(\log U)(Nv)=[\log U,Nv]=N(\log U(v))$,
hence ${\rm ad}(\log U)^k(Nv)=N((\log U)^k v)$, $k\ge 0$. Then
\begin{eqnarray*}
U(Nv)U^{-1} &=& e^{\log U} \circ Nv \circ e^{-\log U}=
\sum_{k=0}^{\infty} \frac{1}{k!}{\rm ad}(\log U)^k (Nv) \\
&=& \sum_{k=0}^{\infty} \frac{1}{k!}N((\log U)^k v)=N(e^{\log U} v)=N(Uv).
\end{eqnarray*}
This completes the proof.
\end{proof}

\subsection{The Goldman Lie algebra and Kontsevich's
``associative"}
Let $\hat{\pi}$ be the set of conjugacy classes of $\pi$.
In other words, $\hat{\pi}$ is the set of free homotopy classes
of oriented loops on $\Sigma$. The Goldman Lie algebra \cite{Go}
of $\Sigma$ is the vector space $\mathbb{Q}\hat{\pi}$ spanned by $\pi$,
equipped with the bracket defined as follows.
Let $\alpha,\beta$ be immersed loops on $\Sigma$
such that their intersections consist of transverse double points.
For each $p\in \alpha \cap \beta$, let $|\alpha_p\beta_p|$
be the free homotopy class of the loop first going the oriented
loop $\alpha$ based at $p$, then going $\beta$ based at $p$.
Also let $\varepsilon(p;\alpha,\beta)\in \{ \pm 1\}$ be the
local intersection number of $\alpha$ and $\beta$ at $p$.
Then the bracket of $\alpha$ and $\beta$ is given by
$$[\alpha, \beta]:=\sum_{p\in \alpha \cap \beta}
\varepsilon(p;\alpha,\beta) |\alpha_p\beta_p| \in
\mathbb{Q}\hat{\pi}.$$
Remark that if the loops $\alpha$ and $\beta$ are
freely homotopic to disjoint curves, then $[\alpha,\beta]=0$.

We recall a main result of \cite{KK}, which relates the Goldman Lie
algebra of $\Sigma$ and Kontsevich's ``associative". Actually, we can give
a slightly generalized form by virtue of Proposition \ref{compatibleU}.

\begin{thm}[\cite{KK} Theorem 1.2.1]
\label{GLAass}
Let $\theta$ be a Magnus expansion of $\pi$
satisfying $\theta(\zeta)=\exp(\omega)$. Then the map
$$-N\theta \colon \mathbb{Q}\hat{\pi} \to N(\widehat{T}_1)=\mathfrak{a}_g^-,
\pi \ni x \mapsto -N\theta(x) \in N(\widehat{T}_1)$$
is a well-defined Lie algebra homomorphism. The kernel is the subspace
$\mathbb{Q}1$ spanned by the constant loop $1$, and the image is dense
in $N(\widehat{T}_1)=\mathfrak{a}_g^-$ with respect to the
$\widehat{T}_1$-adic topology.
\end{thm}

\begin{proof}
If $\theta$ is symplectic, this is \cite{KK} Theorem 1.2.1.
We just remark that in the proof of \cite{KK} Theorem 1.2.1,
we use (co)homology theory of Hopf algebras, hence we need
$\theta$ to be group-like.

Fix a symplectic expansion $\theta$ and
let $\theta^{\prime}$ be a Magnus expansion satisfying
$\theta^{\prime}(\zeta)=\exp(\omega)$. Then there exists
$U\in {\rm IA}_{\omega}(\widehat{T})$ such that $\theta^{\prime}=
U\circ \theta$ (see \cite{KK} \S2.8). The map
$\mathfrak{a}_g^-\to \mathfrak{a}_g^-,\ D\mapsto U\circ D\circ U^{-1}$
is a Lie algebra automorphism, and for any $x\in \pi$ we have
$$U\circ (-N\theta(x))\circ U^{-1}=-N(U\theta(x))=-N\theta^{\prime}(x),$$
by Proposition \ref{compatibleU}. This completes the proof.
\end{proof}

There is an action of the mapping class group $\mathcal{M}_{g,1}$
on the Goldman Lie algebra $\mathbb{Q}\hat{\pi}$. The action
is induced from the action on $\pi$.

\begin{thm}
\label{Mgaction}
Let $\theta$ be a Magnus expansion of $\pi$
satisfying $\theta(\zeta)=\exp(\omega)$.
Then the mapping class group $\mathcal{M}_{g,1}$ acts
on $\mathfrak{a}_g^-$ by $f\cdot D=T^{\theta}(f) \circ D \circ T^{\theta}(f)^{-1}$,
where $f\in \mathcal{M}_{g,1}$, $D\in \mathfrak{a}_g^-$.
Moreover, the Lie algebra homomorphism
$-N\theta \colon \mathbb{Q}\hat{\pi} \to \mathfrak{a}_g^-$
in Theorem \ref{GLAass} is
$\mathcal{M}_{g,1}$-equivariant.
\end{thm}

\begin{proof}
It suffices to prove that $-N\theta \colon \mathbb{Q}\hat{\pi} \to \mathfrak{a}_g^-$
is $\mathcal{M}_{g,1}$-equivariant. Let $x\in \pi$ and $f\in \mathcal{M}_{g,1}$.
Since $\theta(\zeta)=e^{\omega}$, the total Johnson map $T^{\theta}(f)$ satisfies
$T^{\theta}(f)e^{\omega}=T^{\theta}(f)\theta(\zeta)
=\theta(f(\zeta))=\theta(\zeta)=e^{\omega}$, hence
$T^{\theta}(f)\omega=\omega$.
By Proposition \ref{compatibleU}, we have
$$-N\theta(f(x)) = -N(T^{\theta}(f)\theta(x))
=T^{\theta}(f) \circ (-N\theta(x)) \circ T^{\theta}(f)^{-1}.$$
This completes the proof.
\end{proof}

\section{A generalization of Dehn twists}

In this section we first recall another main result of \cite{KK}, which
describes the total Johnson map of the Dehn twist
along a simple closed curve. Motivated by this result,
we introduce a generalization of Dehn twists for not
necessarily simple loops on $\Sigma$, as automorphisms
of the completed group ring of $\pi$.

\subsection{The logarithms of Dehn twists}
For a Magnus expansion $\theta$ of $\pi$, we denote
$\ell^{\theta}:=\log \theta$. This is a map from
$\pi$ to $\widehat{T}_1$. Note that the
logarithm is defined on the set $1+\widehat{T}_1$.

\begin{dfn}
For $x\in \pi$, set
$$L^{\theta}(x):=\frac{1}{2}N(\ell^{\theta}(x)
\ell^{\theta}(x)) \in \widehat{T}_2.$$
\end{dfn}

In fact, $L^{\theta}$ descends to an invariant
of unoriented loops on $\Sigma$.

\begin{lem}[\cite{KK} Lemma 2.6.4]
For $x,y\in \pi$, we have
\begin{enumerate}
\item $L^{\theta}(x^{-1})=L^{\theta}(x)$,
\item $L^{\theta}(yxy^{-1})=L^{\theta}(x)$.
\end{enumerate}
\end{lem}

Hence for any unoriented loop $\gamma \subset \Sigma$, we can define
$L^{\theta}(\gamma)$ as $L^{\theta}(x)$, where $x\in \pi$
is freely homotopic to $\gamma$.
As in \S2.2, we regard $L^{\theta}(x)\in \widehat{T}_2$
as a derivation on $\widehat{T}$.

For a simple closed curve $C$ on $\Sigma$, we denote by
$t_C\in \mathcal{M}_{g,1}$ the right handed Dehn twist along $C$.
The following was proved in \cite{KK}. In fact what is given
here is a slightly generalized form.

\begin{thm}[\cite{KK}, Theorem 1.1.1]
\label{logDT}
Let $\theta$ be a Magnus expansion of $\pi$ satisfying
$\theta(\zeta)=\exp(\omega)$, and
$C$ a simple closed curve on $\Sigma$. Then
the total Johnson map $T^{\theta}(t_C)$ is
described as
$$T^{\theta}(t_C)=e^{-L^{\theta}(C)}.$$
Here, the right hand side is the algebra
automorphism of $\widehat{T}$ defined by the
exponential of the derivation $-L^{\theta}(C)$.
\end{thm}

\begin{proof}
If $\theta$ is symplectic, this is \cite{KK} Theorem 1.1.1.
Fix a symplectic expansion $\theta$ and let $\theta^{\prime}$
be a Magnus expansion satisfying $\theta^{\prime}(\zeta)=\exp(\omega)$.
As in the proof of Theorem \ref{GLAass}, there exists
$U\in {\rm IA}_{\omega}(\widehat{T})$ such that $\theta^{\prime}=
U\circ \theta$. Note that this condition implies $\ell^{\theta^{\prime}}
=U\circ \ell^{\theta}$. Then for any $\varphi\in \mathcal{M}_{g,1}$, we have
$T^{\theta^{\prime}}(\varphi)=U\circ T^{\theta}(\varphi) \circ U^{-1}$,
since $T^{\theta^{\prime}}(\varphi) \circ \theta^{\prime}
=\theta^{\prime} \circ \varphi=U\circ \theta \circ \varphi
=U\circ T^{\theta}(\varphi)\circ \theta=U\circ T^{\theta}(\varphi)
\circ U^{-1} \circ \theta^{\prime}$ and the linear span of the image $\theta^{\prime}(\pi)$
is dense in $\widehat{T}$ (see \cite{Kaw1}). On the other
hand, for $x\in \pi$ we have
\begin{eqnarray*}
L^{\theta^{\prime}}(x)&=& \frac{1}{2}N(\ell^{\theta^{\prime}}(x)
\ell^{\theta^{\prime}}(x))
=\frac{1}{2}N(U\ell^{\theta}(x)U\ell^{\theta}(x))
=\frac{1}{2}N(U(\ell^{\theta}(x)\ell^{\theta}(x))) \\
&=& U\circ \left( \frac{1}{2}N(\ell^{\theta}(x)\ell^{\theta}(x)) \right) \circ U^{-1}
=U\circ L^{\theta}(x) \circ U^{-1}
\end{eqnarray*}
by Proposition \ref{compatibleU},
hence $e^{-L^{\theta^{\prime}}(C)}=U\circ e^{-L^{\theta}(C)} \circ U^{-1}$.
Now the formula for $\theta^{\prime}$ follows from the formula
for $\theta$.
\end{proof}

\begin{rem}
\label{another}
{\rm Fix a Magnus expansion $\theta$ of $\pi$ satisfying $\theta(\zeta)=\exp(\omega)$.
Let $f\in \mathcal{M}_{g,1}$ and $C$ a simple closed curve on $\Sigma$.
By an argument similar to the proof of $L^{\theta^{\prime}}(x)
=U\circ L^{\theta}(x) \circ U^{-1}$ above, we have
$L^{\theta}(f(C))=T^{\theta}(f) \circ L^{\theta}(C) \circ T^{\theta}(f)^{-1}$.
Therefore we might expect a possibility of another proof of \cite{KK} Theorem 1.1.1.
For example, let us restrict our attentions to non-separating simple closed curves.
Any two non-separating simple closed curves are in the same orbit
of the action of $\mathcal{M}_{g,1}$ on the set of unoriented loops
(up to homotopy) on $\Sigma$.
If one could prove the formula $T^{\theta}(t_C)=e^{-L^{\theta}(C)}$
for a particular choice of a Magnus expansion $\theta$ satisfying
$\theta(\zeta)=\exp(\omega)$, say one of the symplectic expansions
in \cite{Kaw2} \cite{Ku} \cite{Mas}, and a particular choice of a simple closed
curve $C$, then the formula for any such $\theta$ and any $C$ non-separating
follows.
}
\end{rem}

\subsection{Generalized Dehn twists}
Now we introduce generalized Dehn twists.
Let $\theta$ be a Magnus expansion of $\pi$ satisfying $\theta(\zeta)=\exp(\omega)$.
We denote by ${\rm Aut}(\widehat{\mathbb{Q}\pi})$ the group of
filter-preserving algebra automorphisms of $\widehat{\mathbb{Q}\pi}$,
which is isomorphic to ${\rm Aut}(\widehat{T})$ as a group through $\theta$.

Let $\gamma$ be an unoriented loop on $\Sigma$.
Then the exponential $e^{-L^{\theta}(\gamma)}$ of the
derivation $-L^{\theta}(\gamma)$ is a filter-preserving
algebra automorphism of $\widehat{T}$. Thus the map
$$t_{\gamma}:=\theta^{-1}\circ e^{-L^{\theta}(\gamma)} \circ \theta$$
lies in ${\rm Aut}(\widehat{\mathbb{Q}\pi})$.
As we see in the proof of Theorem \ref{logDT},
if $\theta^{\prime}$ is another Magnus expansion
satisfying $\theta^{\prime}(\zeta)=\exp(\omega)$
then $L^{\theta^{\prime}}(\gamma)=U \circ L^{\theta}(\gamma)\circ U^{-1}$,
where $U\in {\rm Aut}_{\omega}(\widehat{T})$ satisfies
$\theta^{\prime}=U\circ \theta$. This shows that $t_{\gamma}$
is actually independent of the choice of $\theta$.

\begin{dfn}
\label{GDT}
We call $t_{\gamma}\in {\rm Aut}(\widehat{\mathbb{Q}\pi})$
the generalized Dehn twist along $\gamma$.
\end{dfn}

We remark that the generalized Dehn twists have the following
natural property. Let $f\in \mathcal{M}_{g,1}$. Then we have
$L^{\theta}(f(\gamma))=T^{\theta}(f)\circ
L^{\theta}(\gamma) \circ T^{\theta}(f)^{-1}$
(see Remark \ref{another}), thus
$$t_{f(\gamma)}=f \circ t_{\gamma} \circ f^{-1}.$$

Since we have a natural injective homomorphism
$$\mathcal{M}_{g,1} \hookrightarrow {\rm Aut}(\pi)
\hookrightarrow {\rm Aut}(\widehat{\mathbb{Q}\pi}),$$
we can ask whether $t_{\gamma}$ gives an element of
$\mathcal{M}_{g,1}$.

\begin{dfn}
\label{defmap}
Let $\gamma$ be an unoriented loop on $\Sigma$.
We say $t_{\gamma}\in {\rm Aut}(\widehat{\mathbb{Q}\pi})$
is a mapping class if $t_{\gamma}$ lies in the
image of $\mathcal{M}_{g,1}$.
\end{dfn}

For example, for any simple closed curve $C$,
the generalized Dehn twist $t_C$ is a mapping class,
and it is the usual right handed Dehn twist
along $C$ by Theorem \ref{logDT}.
Moreover, for any power of $C$, the generalized
Dehn twist along it is a mapping class. For, we have
$L^{\theta}(C^m)=m^2L^{\theta}(C)$
hence $t_{C^m}=(t_C)^{m^2}$, where $m\in \mathbb{Z}$.

\subsection{The support of a generalized Dehn twist}

We shall give a criterion for the realizability of
$t_{\gamma}$ as a mapping class.
We use the following, which would be well-known
to experts. Similar statements can be
found in several literatures, but we do not find a
suitable reference.

\begin{lem}
\label{folklore}
Let $S$ be a compact connected oriented surface, and
$\varphi$ an orientation preserving diffeomorphism of
$S$ fixing the boundary $\partial S$ pointwise. If
$\varphi$ preserves every homotopy class of oriented
loops on $S$, then $\varphi$ is isotopic relative to the boundary
$\partial S$ to a product of boundary-parallel Dehn twists.
Here a boundary-parallel
Dehn twist is meant a Dehn twist along a simple
closed curve which is parallel to one of the components of $\partial S$.
\end{lem}

\begin{proof}
Let $g$ be the genus of $S$ and $r$ the number of the boundary components
of $S$. In case $S$ is the 2-sphere or a disk or an annulus,
the assertion is trivial. Thus we may assume if $g=0$, $r\ge 3$.
Then we can choose a collection $\{ C_i\}_i$ of $2g+r$ (if $g=0$, then $r-1$)
simple closed curves on $S$ satisfying the three properties from \cite{FM} Proposition 2.8,
see Figure 3. By \cite{FM} Lemma 2.9, we can deform $\varphi$ by an isotopy
into a diffeomorphism fixing the union $\bigcup_i C_i$ pointwise.
The complement $S\setminus \bigcup_i C_i$ is a disjoint union of
three disks and $r$ annuli. The restriction of $\varphi$ to each
disk component is isotopic relative to the boundary to the identity,
and the restriction to each annulus component is isotopic relative to
the boundary to a power of the Dehn twist along a simple closed curve
parallel to the boundary of the annulus.
Thus we can deform $\varphi$ by an isotopy into
a product of boundary-parallel Dehn twists. This completes the proof.
\end{proof}

\begin{center}
Figure 3: $\{ C_i \}_i$ for $g=2,r=3$

\input{figure3.tex}

\end{center}

\begin{thm}
\label{support}
Let $\gamma$ be an unoriented loop on $\Sigma$ and suppose
the generalized Dehn twist $t_{\gamma}$ is a mapping class.
Then $t_{\gamma}\in \mathcal{M}_{g,1}$ is represented by
a diffeomorphism whose support lies in a regular neighborhood of
$\gamma$.
\end{thm}

\begin{proof}
We claim that if $\delta$ is an oriented loop on $\Sigma$
disjoint from $\gamma$, then $t_{\gamma}(\delta)=\delta$.
Let $x\in \pi$ be a representative of $\delta$.
By Theorem \ref{Mgaction},
\begin{eqnarray*}
-N\theta(t_{\gamma}(x))=
&=& t_{\gamma} \cdot (-N\theta(x)) \\
&=& T^{\theta}(t_{\gamma}) \circ (-N\theta(x)) \circ T^{\theta}(t_{\gamma})^{-1} \\
&=& e^{-L^{\theta}(\gamma)} \circ (-N\theta(x)) \circ e^{L^{\theta}(\gamma)} \\
&=& \sum_{m=0}^{\infty} \frac{1}{m!} {\rm ad}(-L^{\theta}(\gamma))^m(-N\theta(x)).
\end{eqnarray*}
But $-L^{\theta}(\gamma)=(1/2)N((\log \theta(\widetilde{\gamma}))^2)
=\sum_{n=1}^{\infty} a_n N\theta((\widetilde{\gamma}-1)^n)$, where
$\widetilde{\gamma}\in \pi$ is a representative of $\gamma$ and
$\sum_{n=1}^{\infty} a_n (z-1)^n$ is the Taylor expansion of $(1/2)(\log z)^2$
at $z=1$. Since $\gamma$ and $\delta$ are disjoint,
the Goldman bracket $[\gamma^n,\delta]$ is $0$ for $n\ge 0$.
By Theorem \ref{GLAass}, we have $[-N\theta(\widetilde{\gamma}^n),-N\theta(x)]=0$
for $n\ge 0$. Thus we obtain $[-L^{\theta}(\gamma),-N\theta(x)]=0$
and $-N\theta(t_{\gamma}(x))=-N\theta(x)$. By Theorem \ref{GLAass}
we have $t_{\gamma}(x)-x=t_{\gamma}(\delta)-\delta\in \mathbb{Q}1$.
Since the action of
$\mathcal{M}_{g,1}$ on $\mathbb{Q}\hat{\pi}$ preserves the
augmentation $\mathbb{Q}\hat{\pi}\to \mathbb{Q}, \hat{\pi} \ni x \mapsto 1$,
we conclude $t_{\gamma}(\delta)=\delta$. The claim is proved.

Let $N$ be a closed regular neighborhood of $\gamma$.
By the claim, each oriented component of the boundary $\partial N$
is preserved by $t_{\gamma}$. By an isotopy, we may assume $t_{\gamma}$
is represented by a diffeomorphism $\varphi$ fixing $\partial N$ pointwise.
Also, if $\delta$ is an oriented loop
on $\Sigma \setminus {\rm Int}(N)$ then $t_{\gamma}(\delta)$
is ambient isotopic to $\delta$.  Since $\partial N$ is preserved by $t_{\gamma}$,
this ambient isotopy can be chosen to have its support in $\Sigma \setminus {\rm Int}(N)$.
By Lemma \ref{folklore}, applying a suitable isotopy to $\varphi$ we can write
$\varphi=\varphi^{\prime} \circ t_{\zeta}^m$, $m\in \mathbb{Z}$,
where $\varphi^{\prime}$ is a diffefomorphism whose support lies in $N$,
and $t_{\zeta}$ is the Dehn twist along a simple closed curve parallel to $\partial \Sigma$.
It should be remarked that
$\partial(\Sigma \setminus {\rm Int}(N))=\partial \Sigma \amalg \partial N$.
Let $\Sigma^{\prime}$ be the connected component of $\Sigma \setminus {\rm Int}(N)$ containing
$\partial \Sigma$. If $\Sigma^{\prime}$ is an annulus,
then $t_{\zeta}$ is isotopic to a diffeomorphism
whose support lies in $N$, so is $\varphi$.

Finally we show that if $\Sigma^{\prime}$ is not an annulus,
then $m$ should be $0$. We can choose a based loop $y$ on
$\Sigma^{\prime}$, freely homotopic to a simple closed curve,
so that $[y]\in H=H_1(\Sigma;\mathbb{Q})$ is not zero.
In fact, if the genus of $\Sigma^{\prime}$ is $\ge 1$,
this is easy. Even if the genus of $\Sigma^{\prime}$ is zero,
there is at least one boundary component $D\subset \partial \Sigma^{\prime}$
which is non-separating on $\Sigma$,
because $\Sigma^{\prime}$ is not an annulus and $N(\gamma)$ is connected.
We can choose $y$ as a simple based loop on $\Sigma^{\prime}$
representing $D$. Since $[y]$ is primitive, we can take
symplectic basis $A_1,B_1,\ldots,A_g,B_g \in H$ such
that $[y]=A_1$. Now we have $t_{\gamma}(y)=\varphi^{\prime}\circ
t_{\zeta}^m(y)=t_{\zeta}^m(y)$. Choosing a symplectic expansion $\theta$
and applying $\theta$ to this equation, we have
$$e^{-L^{\theta}(\gamma)}\theta(y)=T^{\theta}(t_{\gamma})\theta(y)
=\theta(t_{\gamma}(y))=\theta(t_{\zeta}^m(y))=T^{\theta}(t_{\zeta}^m)\theta(y)
=e^{-mL^{\theta}(\zeta)}\theta(y).$$
Taking the logarithm, we get
$L^{\theta}(\gamma)\theta(y)=mL^{\theta}(\zeta)\theta(y)$. By \cite{KK}
Corollary 6.5.3, we have $L^{\theta}(\gamma)\theta(y)=0$. On the other hand,
modulo $\widehat{T}_4$, we have
$$L^{\theta}(\zeta)\theta(y) \equiv \frac{1}{2}N(\omega \omega)(A_1)
=\omega A_1-A_1\omega\neq 0,$$
thus $m=0$. This completes the proof.
\end{proof}

\section{Loops with a single transverse double point}

In this section we consider a class of loops on a surface,
which we could say the simplest next to simple closed
curves. We will call such a loop a figure eight.
First we consider the case when the surface is a pair of pants.
After that we classify the possible configurations of a figure eight
on $\Sigma$.

\subsection{Figure eight on a surface}

\begin{dfn}
Let $S$ be an oriented surface and $\gamma$
an unoriented immersed loop on $S$. We say $\gamma$
is a figure eight on $S$ if its self-intersection consists of
a single transverse double point and $\gamma$ is not homotopic to a simple closed curve
or a square of a simple closed curve.
\end{dfn}

A simple closed curve or a square of a simple closed
curve can be deformed into an immersed loop
with a single transverse double point.
But in view of the remark after Definition \ref{defmap},
we exclude them from the above definition.

\subsection{Figure eight on a pair of pants}
Let $P$ be a pair of pants. Fixing a base point $*\in {\rm Int}(P)$,
let $\delta_1,\delta_2,\delta_3$ be simple based loops on $P$
such that each $\delta_i$ is freely homotopic to one of the
oriented boundary component of $P$, and $\delta_1\delta_2\delta_3=1\in
\pi_1(P,*)$. See Figure 4. We denote by $A_i$
the boundary component of $P$ which is freely homotopic to $\delta_i$.
Let $\gamma_1$, $\gamma_2$, and $\gamma_3$ be immersed loops on $P$
which are homotopic to $\delta_2\delta_3^{-1}$, $\delta_3\delta_1^{-1}$,
and $\delta_1\delta_2^{-1}$, respectively. The underlying unoriented
loop of $\gamma_i$ is a figure eight.

\begin{center}
Figure 4: a pair of pants

\unitlength 0.1in
\begin{picture}( 20.0000, 20.0000)( 12.0000,-24.0000)
%
\special{pn 20}%
\special{sh 1}%
\special{ar 2200 1400 10 10 0  6.28318530717959E+0000}%
%
\special{pn 13}%
\special{ar 2200 1400 1000 1000  0.0000000 6.2831853}%
%
\special{pn 13}%
\special{ar 1800 1400 200 200  0.0000000 6.2831853}%
%
\special{pn 13}%
\special{ar 2600 1400 200 200  0.0000000 6.2831853}%
%
\special{pn 8}%
\special{ar 2200 1400 900 900  5.5745590 6.2831853}%
\special{ar 2200 1400 900 900  0.0000000 3.8682350}%
%
\special{pn 8}%
\special{pa 1530 810}%
\special{pa 1552 786}%
\special{pa 1572 760}%
\special{pa 1594 736}%
\special{pa 1616 714}%
\special{pa 1640 692}%
\special{pa 1666 672}%
\special{pa 1692 656}%
\special{pa 1722 640}%
\special{pa 1752 628}%
\special{pa 1784 620}%
\special{pa 1816 612}%
\special{pa 1850 608}%
\special{pa 1882 608}%
\special{pa 1912 610}%
\special{pa 1940 616}%
\special{pa 1966 624}%
\special{pa 1990 634}%
\special{pa 2012 648}%
\special{pa 2034 664}%
\special{pa 2052 684}%
\special{pa 2070 704}%
\special{pa 2086 728}%
\special{pa 2100 752}%
\special{pa 2112 780}%
\special{pa 2124 808}%
\special{pa 2136 840}%
\special{pa 2144 872}%
\special{pa 2152 906}%
\special{pa 2160 940}%
\special{pa 2166 978}%
\special{pa 2172 1016}%
\special{pa 2178 1054}%
\special{pa 2182 1094}%
\special{pa 2186 1134}%
\special{pa 2188 1176}%
\special{pa 2192 1218}%
\special{pa 2194 1260}%
\special{pa 2196 1302}%
\special{pa 2198 1346}%
\special{pa 2200 1388}%
\special{pa 2200 1400}%
\special{sp}%
%
\special{pn 8}%
\special{pa 2870 810}%
\special{pa 2850 786}%
\special{pa 2828 760}%
\special{pa 2808 736}%
\special{pa 2784 714}%
\special{pa 2762 692}%
\special{pa 2736 672}%
\special{pa 2708 656}%
\special{pa 2680 640}%
\special{pa 2648 628}%
\special{pa 2616 620}%
\special{pa 2584 612}%
\special{pa 2552 608}%
\special{pa 2520 608}%
\special{pa 2490 610}%
\special{pa 2462 616}%
\special{pa 2436 624}%
\special{pa 2412 634}%
\special{pa 2388 648}%
\special{pa 2368 664}%
\special{pa 2348 684}%
\special{pa 2332 704}%
\special{pa 2316 728}%
\special{pa 2302 752}%
\special{pa 2288 780}%
\special{pa 2276 808}%
\special{pa 2266 840}%
\special{pa 2256 872}%
\special{pa 2248 906}%
\special{pa 2240 940}%
\special{pa 2234 978}%
\special{pa 2228 1016}%
\special{pa 2224 1054}%
\special{pa 2220 1094}%
\special{pa 2216 1134}%
\special{pa 2212 1176}%
\special{pa 2210 1218}%
\special{pa 2208 1260}%
\special{pa 2204 1302}%
\special{pa 2202 1346}%
\special{pa 2202 1388}%
\special{pa 2200 1400}%
\special{sp}%
%
\special{pn 8}%
\special{ar 1800 1400 350 350  1.5707963 4.7123890}%
%
\special{pn 8}%
\special{pa 1810 1050}%
\special{pa 1844 1050}%
\special{pa 1874 1050}%
\special{pa 1904 1054}%
\special{pa 1934 1062}%
\special{pa 1960 1074}%
\special{pa 1986 1090}%
\special{pa 2010 1108}%
\special{pa 2032 1130}%
\special{pa 2054 1154}%
\special{pa 2074 1180}%
\special{pa 2094 1206}%
\special{pa 2112 1236}%
\special{pa 2130 1268}%
\special{pa 2148 1300}%
\special{pa 2166 1332}%
\special{pa 2184 1366}%
\special{pa 2200 1400}%
\special{pa 2200 1400}%
\special{sp}%
%
\special{pn 8}%
\special{pa 2590 1050}%
\special{pa 2558 1050}%
\special{pa 2526 1050}%
\special{pa 2496 1054}%
\special{pa 2468 1062}%
\special{pa 2442 1074}%
\special{pa 2416 1090}%
\special{pa 2392 1108}%
\special{pa 2370 1130}%
\special{pa 2348 1154}%
\special{pa 2328 1180}%
\special{pa 2308 1206}%
\special{pa 2288 1236}%
\special{pa 2270 1268}%
\special{pa 2252 1300}%
\special{pa 2234 1332}%
\special{pa 2218 1366}%
\special{pa 2200 1400}%
\special{pa 2200 1400}%
\special{sp}%
%
\special{pn 8}%
\special{pa 2590 1750}%
\special{pa 2558 1752}%
\special{pa 2526 1750}%
\special{pa 2496 1746}%
\special{pa 2468 1738}%
\special{pa 2442 1726}%
\special{pa 2416 1712}%
\special{pa 2392 1692}%
\special{pa 2370 1672}%
\special{pa 2348 1648}%
\special{pa 2328 1622}%
\special{pa 2308 1594}%
\special{pa 2288 1564}%
\special{pa 2270 1534}%
\special{pa 2252 1502}%
\special{pa 2234 1468}%
\special{pa 2218 1434}%
\special{pa 2200 1400}%
\special{pa 2200 1400}%
\special{sp}%
%
\special{pn 8}%
\special{pa 1810 1750}%
\special{pa 1844 1752}%
\special{pa 1874 1750}%
\special{pa 1904 1746}%
\special{pa 1934 1738}%
\special{pa 1960 1726}%
\special{pa 1986 1712}%
\special{pa 2010 1692}%
\special{pa 2032 1672}%
\special{pa 2054 1648}%
\special{pa 2074 1622}%
\special{pa 2094 1594}%
\special{pa 2112 1564}%
\special{pa 2130 1534}%
\special{pa 2148 1502}%
\special{pa 2166 1468}%
\special{pa 2184 1434}%
\special{pa 2200 1400}%
\special{pa 2200 1400}%
\special{sp}%
%
\special{pn 8}%
\special{ar 2600 1400 350 350  4.7123890 6.2831853}%
\special{ar 2600 1400 350 350  0.0000000 1.5707963}%
%
\special{pn 8}%
\special{pa 1800 1050}%
\special{pa 1900 980}%
\special{fp}%
%
\special{pn 8}%
\special{pa 1800 1050}%
\special{pa 1900 1120}%
\special{fp}%
%
\special{pn 8}%
\special{pa 2600 1750}%
\special{pa 2500 1680}%
\special{fp}%
%
\special{pn 8}%
\special{pa 2600 1750}%
\special{pa 2500 1820}%
\special{fp}%
%
\special{pn 8}%
\special{pa 2530 610}%
\special{pa 2430 550}%
\special{fp}%
\special{pa 2470 570}%
\special{pa 2470 570}%
\special{fp}%
%
\special{pn 8}%
\special{pa 2530 610}%
\special{pa 2430 670}%
\special{fp}%
\special{pa 2470 650}%
\special{pa 2470 650}%
\special{fp}%
\put(17.1000,-9.6000){\makebox(0,0)[lb]{$\delta_1$}}%
\put(25.1000,-19.6000){\makebox(0,0)[lb]{$\delta_2$}}%
\put(24.3000,-8.5000){\makebox(0,0)[lb]{$\delta_3$}}%
\put(21.6000,-16.2000){\makebox(0,0)[lb]{$*$}}%
\end{picture}%

\end{center}

\begin{lem}
\label{Lem8inP}
Let $\gamma$ be a figure eight on $P$, and $N(\gamma)\subset P$
a closed regular neighborhood of $\gamma$. Then
the complement $P\setminus {\rm Int}(N(\gamma))$ is a disjoint union
of three annuli. Each annulus contains
one boundary component of $P$ and one boundary
component of $N(\gamma)$. In particular, the inclusion
$N(\gamma) \subset P$ is a strong deformation retract.
\end{lem}

\begin{proof}
The regular neighborhood $N(\gamma)$ is diffeomorphic
to a pair of pants. Cutting at the double point,
we can divide the loop $\gamma$  into two simple closed curves. Let
$C_1$ and $C_2$ be the boundary components of
$N(\gamma)$ which are homotopic to these simple closed curves,
and $C_3$ the remaining boundary component of $N(\gamma)$.
The complement $P\setminus {\rm Int}(N(\gamma))$ has
six boundary components, $\{ A_i,C_i \}$.
The Euler characteristic of $P\setminus {\rm Int}(N(\gamma))$
is equal to $\chi(P)-\chi(N(\gamma))=0$, and
each connected component has genus $0$.

We denote by $\Sigma_{0,r}$ the compact connected
oriented surface of genus $0$ with $r$ boundary components.
Computing the Euler characteristic, we see that
the possible topological types of $P\setminus {\rm Int}(N(\gamma))$
are $\Sigma_{0,4} \amalg \Sigma_{0,1} \amalg \Sigma_{0,1}$,
$\Sigma_{0,3} \amalg \Sigma_{0,2} \amalg \Sigma_{0,1}$,
or $\Sigma_{0,2} \amalg \Sigma_{0,2} \amalg \Sigma_{0,2}$.
For, if $P\setminus {\rm Int}(N(\gamma))\cong \Sigma_{0,r_1} \amalg
\cdots \amalg \Sigma_{0,r_k}$, we must have
$r_i \le 6$ and $\sum_{i} (2-r_i)=0$. Suppose
$P\setminus {\rm Int}(N(\gamma))\cong
\Sigma_{0,4} \amalg \Sigma_{0,1} \amalg \Sigma_{0,1}$.
Since $A_i$ are not homologous to $0$ as a homology
class of $P$, $A_i$ does not bound a disk. Also,
if $C_1$ or $C_2$ bounds a disk, then $\gamma$ is homotopic
to a simple closed curve. Therefore neither $C_1$ nor $C_2$
bounds a disk. But there are two disk components
in $P\setminus {\rm Int}(N(\gamma))$, a contradiction.
Suppose $P\setminus {\rm Int}(N(\gamma))\cong
\Sigma_{0,3} \amalg \Sigma_{0,2} \amalg \Sigma_{0,1}$.
From what we have just seen, the boundary of $\Sigma_{0,1}$
must be $C_3$. Since $A_i$ are not homologous to each other,
any two of them do not appear in the boundary of
$\Sigma_{0,2}$. Therefore, by a suitable renumbering,
we may assume the boundary of $\Sigma_{0,2}$
is $C_1 \amalg C_2$ or $A_1 \amalg C_1$. If
$\partial \Sigma_{0,2} \cong C_1 \amalg C_2$, then we can
construct a simple closed curve in $\Sigma_{0,2} \cup N(\gamma)
\subset P$ meeting $C_1$ in
a transverse double point. This contradicts to
the fact that the genus of $P$ is $0$.
If $\partial \Sigma_{0,2} \cong A_1 \amalg C_1$,
$\gamma$ is contained in the annulus
$(N(\gamma) \cup \Sigma_{0,1}) \cup (\Sigma_{0,2})$,
which contains $A_1$ as a strong deformation retract.
This implies that $\gamma$ is homotopic to a power of
the simple closed curve $A_1$, a contradiction.
Therefore, we have $P\setminus {\rm Int}(N(\gamma))
\cong \Sigma_{0,2} \amalg \Sigma_{0,2} \amalg \Sigma_{0,2}$.

The remaining part of the lemma follows from the
fact $A_i$ are not homologous to each other.
\end{proof}

\begin{prop}
\label{8inP}
Any figure eight on $P$ is isotopic relative to the
boundary $\partial P$ to one of
$\gamma_1$, $\gamma_2$, or $\gamma_3$. 
\end{prop}

\begin{proof}
Let $\gamma$ be a figure eight on $P$ and $p$ the
unique self intersection of $\gamma$. By an isotopy,
we may assume $p=*$. Let $\gamma
\colon [0,1]\to P$ be a parametrization of $\gamma$ satisfying
$\gamma(0)=\gamma(1/2)=\gamma(1)=*$.
Let $f_1:=\gamma|_{[0,1/2]}$ and $f_2:=\gamma|_{[1/2,1]}$, and
consider the map $f:=f_1 \vee f_2\colon S^1 \vee S^1 \to P$.
Since $\gamma$ is not homotopic to a simple closed curve,
the image of $f_i$ are essential simple closed curves on $P$,
thus parallel to one of the boundaries of $P$.
Also, by Lemma \ref{Lem8inP}, the induced map
$f_* \colon \pi_1(S^1 \vee S^1) \to \pi_1(P)$ is
an isomorphism.
By a suitable renumbering and taking the inverse
of the parametrization of $f_1$ or $f_2$,
we may assume $(f_1,f_2)$ are free homotopic to
one of $(\delta_1, \delta_2)$, $(\delta_2,\delta_3)$,
or $(\delta_3,\delta_1)$, respectively.

Suppose $f_1$ and $f_2$ are free homotopic to $\delta_1$ and
$\delta_2$, respectively. Note that $\pi_1(P,*)$ is a free group
of rank two generated by $\delta_1$ and $\delta_2$.
If we regard $f_1$ and $f_2$ as
elements of $\pi_1(P,*)$, the endomorphism defined by
$\delta_1 \mapsto f_1$ and $\delta_2 \mapsto f_2$ is
an isomorphism which act trivially on the abelianization
of $\pi_1(P,*)$. But it is classically known that such
an isomorphism is in fact an inner automorphism,
see \cite{MKS} \S3.5, Corollary N4. Thus there exists
$x\in \pi_1(P,*)$ such that $f_1=x^{-1}\delta_1 x$,
$f_2=x^{-1}\delta_2 x$.

Representing $x$ as a loop based at $*$, we regard
$x$ as an isotopy of $*$. Let $\{\Psi_t \}_{t\in [0,1]}$
be an ambient isotopy of $P$ relative to $\partial P$ extending $x$.
Then $\Psi_1(f_1)=\delta_1$ and $\Psi_1(f_2)=\delta_2$.
This implies $\Psi_1(\gamma)$ is isotopic to $\gamma_3$,
so is $\gamma=\Psi_0(\gamma)$.

By the same way, if $(f_1,f_2)$ are free homotopic to
$(\delta_2,\delta_3)$ or $(\delta_3,\delta_1)$, we
conclude $\gamma$ is isotopic to $\gamma_1$ or
$\gamma_2$, respectively. This completes the proof.
\end{proof}

\subsection{Figure eight on the surface $\Sigma$}

\begin{lem}
\label{PinSig}
Let $Q$ be a closed subset of $\Sigma$ which
lies in ${\rm Int}(\Sigma)$ and is diffeomormorphic to a pair of pants.
Then the pair $(\Sigma, Q)$ is diffeomorphic
to one of the pairs in Figure 5.
\end{lem}

\begin{center}
Figure 5: the pair $(\Sigma,Q)$

\input{figure5-1.tex}

\vspace{0.6cm}

\unitlength 0.1in
\begin{picture}( 33.0000, 15.6000)( 14.0000,-17.9000)
%
\special{pn 13}%
\special{ar 4600 1200 100 400  0.0000000 6.2831853}%
%
\special{pn 13}%
\special{ar 1800 1200 400 400  1.5707963 4.7123890}%
%
\special{pn 13}%
\special{ar 1600 1200 100 100  0.0000000 6.2831853}%
%
\special{pn 13}%
\special{ar 2000 1200 100 100  0.0000000 6.2831853}%
\put(17.0000,-12.5000){\makebox(0,0)[lb]{$\cdots$}}%
%
\special{pn 13}%
\special{ar 3950 1200 100 100  0.0000000 6.2831853}%
%
\special{pn 13}%
\special{ar 4350 1200 100 100  0.0000000 6.2831853}%
\put(40.5000,-12.5000){\makebox(0,0)[lb]{$\cdots$}}%
%
\special{pn 13}%
\special{ar 2760 970 100 100  0.0000000 6.2831853}%
%
\special{pn 13}%
\special{ar 3160 970 100 100  0.0000000 6.2831853}%
\put(28.6000,-10.2000){\makebox(0,0)[lb]{$\cdots$}}%
%
\special{pn 8}%
\special{ar 2200 1200 50 400  1.5707963 4.7123890}%
%
\special{pn 8}%
\special{ar 2200 1200 50 400  4.7123890 4.9790556}%
\special{ar 2200 1200 50 400  5.1390556 5.4057223}%
\special{ar 2200 1200 50 400  5.5657223 5.8323890}%
\special{ar 2200 1200 50 400  5.9923890 6.2590556}%
\special{ar 2200 1200 50 400  6.4190556 6.6857223}%
\special{ar 2200 1200 50 400  6.8457223 7.1123890}%
\special{ar 2200 1200 50 400  7.2723890 7.5390556}%
\special{ar 2200 1200 50 400  7.6990556 7.8539816}%
%
\special{pn 8}%
\special{ar 3750 1200 50 400  1.5707963 4.7123890}%
%
\special{pn 8}%
\special{ar 3750 1200 50 400  4.7123890 4.9790556}%
\special{ar 3750 1200 50 400  5.1390556 5.4057223}%
\special{ar 3750 1200 50 400  5.5657223 5.8323890}%
\special{ar 3750 1200 50 400  5.9923890 6.2590556}%
\special{ar 3750 1200 50 400  6.4190556 6.6857223}%
\special{ar 3750 1200 50 400  6.8457223 7.1123890}%
\special{ar 3750 1200 50 400  7.2723890 7.5390556}%
\special{ar 3750 1200 50 400  7.6990556 7.8539816}%
%
\special{pn 13}%
\special{pa 1800 800}%
\special{pa 4600 800}%
\special{fp}%
%
\special{pn 13}%
\special{pa 1800 1600}%
\special{pa 4600 1600}%
\special{fp}%
%
\special{pn 8}%
\special{ar 2950 1100 350 50  6.2831853 6.2831853}%
\special{ar 2950 1100 350 50  0.0000000 3.1415927}%
%
\special{pn 8}%
\special{ar 3350 1100 50 300  3.1415927 4.7123890}%
%
\special{pn 8}%
\special{ar 2550 1100 50 300  4.7123890 6.2831853}%
%
\special{pn 8}%
\special{ar 2550 1100 50 300  3.1415927 3.4844498}%
\special{ar 2550 1100 50 300  3.6901641 4.0330212}%
\special{ar 2550 1100 50 300  4.2387355 4.5815927}%
%
\special{pn 8}%
\special{ar 3350 1100 50 300  4.7123890 5.0552461}%
\special{ar 3350 1100 50 300  5.2609604 5.6038176}%
\special{ar 3350 1100 50 300  5.8095318 6.1523890}%
%
\special{pn 8}%
\special{ar 2950 1100 450 150  6.2831853 6.4831853}%
\special{ar 2950 1100 450 150  6.6031853 6.8031853}%
\special{ar 2950 1100 450 150  6.9231853 7.1231853}%
\special{ar 2950 1100 450 150  7.2431853 7.4431853}%
\special{ar 2950 1100 450 150  7.5631853 7.7631853}%
\special{ar 2950 1100 450 150  7.8831853 8.0831853}%
\special{ar 2950 1100 450 150  8.2031853 8.4031853}%
\special{ar 2950 1100 450 150  8.5231853 8.7231853}%
\special{ar 2950 1100 450 150  8.8431853 9.0431853}%
\special{ar 2950 1100 450 150  9.1631853 9.3631853}%
\put(16.0000,-4.0000){\makebox(0,0)[lb]{case 4 ($k_1,k_2,h \ge 0, k_1+k_2+h=g$)}}%
\put(28.0000,-18.3000){\makebox(0,0)[lb]{$Q$}}%
\put(14.8000,-18.8000){\makebox(0,0)[lb]{$\underbrace{\hspace{1.5cm}}_{\ }$}}%
\put(17.0000,-19.6000){\makebox(0,0)[lb]{$k_1$}}%
\put(38.8000,-18.8000){\makebox(0,0)[lb]{$\underbrace{\hspace{1.5cm}}_{\ }$}}%
\put(41.0000,-19.6000){\makebox(0,0)[lb]{$h$}}%
\put(26.2000,-7.7000){\makebox(0,0)[lb]{$\overbrace{\hspace{1.5cm}}^{\ }$}}%
\put(28.8000,-6.1000){\makebox(0,0)[lb]{$k_2$}}%
\end{picture}%

\end{center}

\begin{proof}
The assertion is obtained by the classification theorem of surfaces.
Consider the complement $\Sigma^{\prime}:=\Sigma \setminus {\rm Int}(Q)$.
This is a compact oriented surface with four boundary components.
The number of connected components is $1$, $2$, or $3$.
Let $\Sigma_{h,r}$ be a compact connected oriented surface
of genus $h$ with $r$ boundary components.
By computing the Euler characteristic of each component,
we can determine the topological types of $\Sigma^{\prime}$.
If $|\pi_0(\Sigma^{\prime})|=1$, then $\Sigma^{\prime}
\cong \Sigma_{g,4}$. This is the case 1. If $|\pi_0(\Sigma^{\prime})|=2$,
$\Sigma^{\prime} \cong \Sigma_{h-1,2} \amalg \Sigma_{g-h,2}$
or $\Sigma^{\prime} \cong \Sigma_{h-1} \amalg \Sigma_{g-h,3}$
where $1\le h\le g$ and $\Sigma_{g-h,2}$ or $\Sigma_{g-h,3}$
is the connected component containing $\partial \Sigma$.
This is the case 2 or 3, respectively.
If $|\pi_0(\Sigma^{\prime})|=3$, then
$\Sigma^{\prime} \cong \Sigma_{k_1,1} \amalg \Sigma_{k_2,1}
\amalg \Sigma_{h,2}$, where $k_1,k_2,h \ge 0$,
$k_1+k_2+h=g$, and $\Sigma_{h,2}$ is the connected component
containing $\partial \Sigma$. This is the case 4.
\end{proof}

\begin{center}
Figure 6: the pair $(\Sigma,\gamma)$

\unitlength 0.1in
\begin{picture}( 25.0000, 11.7000)( 16.0000,-14.0000)
%
\special{pn 13}%
\special{ar 4000 1000 100 400  0.0000000 6.2831853}%
%
\special{pn 13}%
\special{pa 4000 600}%
\special{pa 2000 600}%
\special{fp}%
%
\special{pn 13}%
\special{ar 2000 1000 400 400  1.5707963 4.7123890}%
%
\special{pn 13}%
\special{pa 4000 1400}%
\special{pa 2000 1400}%
\special{fp}%
%
\special{pn 13}%
\special{ar 2000 1000 100 100  0.0000000 6.2831853}%
%
\special{pn 13}%
\special{ar 2530 1000 100 100  0.0000000 6.2831853}%
%
\special{pn 13}%
\special{ar 3000 1000 100 100  0.0000000 6.2831853}%
%
\special{pn 13}%
\special{ar 3700 1000 100 100  0.0000000 6.2831853}%
\put(31.0000,-10.5000){\makebox(0,0)[lb]{$\cdots$}}%
\put(18.0000,-4.0000){\makebox(0,0)[lb]{case I}}%
%
\special{pn 8}%
\special{ar 2000 1000 170 170  0.0000000 6.2831853}%
%
\special{pn 8}%
\special{ar 2530 1000 170 170  0.0000000 6.2831853}%
%
\special{pn 8}%
\special{ar 2530 1000 320 320  4.7123890 4.8998890}%
\special{ar 2530 1000 320 320  5.0123890 5.1998890}%
\special{ar 2530 1000 320 320  5.3123890 5.4998890}%
\special{ar 2530 1000 320 320  5.6123890 5.7998890}%
\special{ar 2530 1000 320 320  5.9123890 6.0998890}%
\special{ar 2530 1000 320 320  6.2123890 6.3998890}%
\special{ar 2530 1000 320 320  6.5123890 6.6998890}%
\special{ar 2530 1000 320 320  6.8123890 6.9998890}%
\special{ar 2530 1000 320 320  7.1123890 7.2998890}%
\special{ar 2530 1000 320 320  7.4123890 7.5998890}%
\special{ar 2530 1000 320 320  7.7123890 7.8539816}%
%
\special{pn 8}%
\special{ar 2000 1000 320 320  1.5707963 1.7582963}%
\special{ar 2000 1000 320 320  1.8707963 2.0582963}%
\special{ar 2000 1000 320 320  2.1707963 2.3582963}%
\special{ar 2000 1000 320 320  2.4707963 2.6582963}%
\special{ar 2000 1000 320 320  2.7707963 2.9582963}%
\special{ar 2000 1000 320 320  3.0707963 3.2582963}%
\special{ar 2000 1000 320 320  3.3707963 3.5582963}%
\special{ar 2000 1000 320 320  3.6707963 3.8582963}%
\special{ar 2000 1000 320 320  3.9707963 4.1582963}%
\special{ar 2000 1000 320 320  4.2707963 4.4582963}%
\special{ar 2000 1000 320 320  4.5707963 4.7123890}%
%
\special{pn 8}%
\special{pa 2000 680}%
\special{pa 2530 680}%
\special{da 0.070}%
%
\special{pn 8}%
\special{pa 2000 1320}%
\special{pa 2530 1320}%
\special{da 0.070}%
%
\special{pn 8}%
\special{ar 2000 1000 250 250  1.5707963 4.7123890}%
%
\special{pn 8}%
\special{ar 2520 1000 250 250  4.7123890 6.2831853}%
\special{ar 2520 1000 250 250  0.0000000 1.5707963}%
%
\special{pn 8}%
\special{pa 2000 1250}%
\special{pa 2032 1252}%
\special{pa 2064 1252}%
\special{pa 2096 1246}%
\special{pa 2128 1234}%
\special{pa 2158 1216}%
\special{pa 2184 1196}%
\special{pa 2204 1170}%
\special{pa 2218 1144}%
\special{pa 2226 1112}%
\special{pa 2232 1082}%
\special{pa 2236 1048}%
\special{pa 2240 1016}%
\special{pa 2246 984}%
\special{pa 2258 954}%
\special{pa 2272 924}%
\special{pa 2290 898}%
\special{pa 2310 872}%
\special{pa 2332 850}%
\special{pa 2356 828}%
\special{pa 2384 810}%
\special{pa 2412 796}%
\special{pa 2440 782}%
\special{pa 2470 768}%
\special{pa 2500 758}%
\special{pa 2520 750}%
\special{sp}%
%
\special{pn 8}%
\special{pa 2520 1250}%
\special{pa 2488 1252}%
\special{pa 2456 1252}%
\special{pa 2424 1246}%
\special{pa 2392 1234}%
\special{pa 2364 1216}%
\special{pa 2338 1196}%
\special{pa 2316 1170}%
\special{pa 2302 1144}%
\special{pa 2294 1112}%
\special{pa 2290 1082}%
\special{pa 2286 1048}%
\special{pa 2282 1016}%
\special{pa 2274 984}%
\special{pa 2264 954}%
\special{pa 2250 924}%
\special{pa 2232 898}%
\special{pa 2212 872}%
\special{pa 2190 850}%
\special{pa 2164 828}%
\special{pa 2138 810}%
\special{pa 2110 796}%
\special{pa 2080 782}%
\special{pa 2050 768}%
\special{pa 2020 758}%
\special{pa 2000 750}%
\special{sp}%
%
\special{pn 13}%
\special{sh 1}%
\special{ar 2260 940 10 10 0  6.28318530717959E+0000}%
%
\special{pn 8}%
\special{pa 2070 1250}%
\special{pa 2110 1180}%
\special{fp}%
\special{pa 2070 1250}%
\special{pa 2140 1270}%
\special{fp}%
\put(20.1000,-15.7000){\makebox(0,0)[lb]{$\gamma$}}%
\end{picture}%

\vspace{0.2cm}

\input{figure6-2.tex}

\vspace{0.2cm}

\input{figure6-3.tex}

\end{center}

\begin{prop}
\label{8inSigma}
Let $\gamma$ be a figure eight on $\Sigma$. Then the pair
$(\Sigma,\gamma)$ is diffeormorphic to one of the pairs
in Figure 6.
\end{prop}

\begin{proof}
Let $N(\gamma)$ be a closed regular neighborhood of $\gamma$.
By Lemma \ref{PinSig}, the pair $(\Sigma,N(\gamma))$ is diffeomorphic
to one of the cases in Figure 5.
If we identify $N(\gamma)$ with $P$, then $\gamma \subset N(\gamma)\cong P$
is isotopic to one of the $\gamma_i$, by Proposition \ref{8inP}.
Note that the choice of $\gamma_i$ corresponds to the choice of
two of the boundary components of $P$.
As we did in Lemma \ref{Lem8inP}, let $C_1$ and $C_2$ be the
boundary components of $N(\gamma)$ which are homotopic to
the simple closed curves obtained by dividing $\gamma$ at
its unique self intersection point.

In the case 1, the curves $C_i$ are all non-separating. Therefore,
we can arrange that $C_1$ and $C_2$ correspond to undotted circles
in the case 1. This is the case I. In the case 2, if $C_1$ and $C_2$
are non-separating, this is the case II-a. The other case
is the case II-a.
The case 3 is treated similarly, and we have the cases III-a or
III-b. Note that in this case $h=1$ is excluded, since
$\gamma$ is not homotopic to a power of a simple closed curve.
In the case 4, if none of $C_1$ or $C_2$ appears on the boundary
of $\Sigma_{h,2}$, this is the case IV-a. The other case
is the case IV-b. Again, to ensure $\gamma$ to be a figure eight,
we need $k_1$ and $k_2$ to be positive.
This completes the proof.
\end{proof}

We shall call a figure eight $\gamma$ on $\Sigma$ is {\it non-separating}
if $\Sigma \setminus {\rm Int}(N(\gamma))$ is connected,
and {\it separating} if $\Sigma \setminus {\rm Int}(N(\gamma))$ is not
connected. In Figure 6, the case I is
non-separating, and the others are separating.

\section{Proof of the main theorem}

In this section we show the main theorem of this paper.

\subsection{Statement and outline of proof}

\begin{thm}
\label{mainthm}
Let $\gamma$ be a figure eight on $\Sigma$.
Then the generalized Dehn twist $t_{\gamma}$
is not a mapping class in the sense of Definition \ref{defmap}.
\end{thm}

Let $\gamma \subset \Sigma$ be a figure eight and $p\in \gamma$
the unique self intersection of $\gamma$. Let
$\gamma \colon [0,1]\to \Sigma$ be a parametrization of $\gamma$
such that $\gamma(0)=\gamma(1/2)=\gamma(1)=p$. Taking a
path $\delta$ from $*\in \partial \Sigma$ to $p$, we denote
$$x:=\delta \cdot \gamma|_{[0,1/2]} \cdot \delta^{-1},\
y:=\delta \cdot (\gamma|_{[1/2,1]})^{-1} \cdot \delta^{-1} \in \pi.$$
Let $N(\gamma)$ be a closed regular neighborhood of $\gamma$. Then
$N(\gamma)$ is diffeomorphic to a pair of pants.
We denote by $C_1$, $C_2$, and $C_3$ the boundary component
of $N(\gamma)$ freely homotopic to $x$, $y$, and $xy$, respectively.
Note that $\gamma$ is freely homotopic to $xy^{-1}$.

For each configuration of a figure eight given in
Proposition \ref{8inSigma}, $x$ and $y$ can be represented
in terms of symplectic generators, as Table 1. Here
we identify the surfaces in Figure 6 with the surface in Figure 2
by a natural way, and the parametrization of $\gamma$ is indicated in Figure 6.

\begin{center}
Table 1

\begin{tabular}{l|c|c}
 & $x$ & $y$ \\
 \hline
 I & $\alpha_1$ & $\alpha_2$ \\
II-a & $(\prod_{i=1}^h [\alpha_i,\beta_i])\beta_h$ & $\beta_h^{-1}$ \\
II-b & $(\prod_{i=1}^h [\alpha_i,\beta_i])\beta_h$ &
$(\prod_{i=1}^h [\alpha_i,\beta_i])^{-1}$ \\
III-a & $(\prod_{i=1}^h [\alpha_i,\beta_i])\beta_h$ &
$\alpha_h\beta_h^{-1}\alpha_h^{-1}$ \\
III-b & $(\prod_{i=1}^h [\alpha_i,\beta_i])\beta_h$ &
$(\prod_{i=1}^{h-1}[\alpha_i,\beta_i])^{-1}$ \\
IV-a & $\prod_{i=1}^{k_1}[\alpha_i,\beta_i]$ &
$\prod_{i=k_1+1}^{k_1+k_2}[\alpha_i,\beta_i]$ \\
IV-b & $\prod_{i=1}^{k_1}[\alpha_i,\beta_i]$ &
$(\prod_{i=1}^{k_1+k_2}[\alpha_i,\beta_i])^{-1}$ \\
\end{tabular}
\end{center}

The proof of Theorem \ref{mainthm} depends on Proposition \ref{8inSigma}
and explicit computations of the invariant $L^{\theta}$ for $x,y,xy$, and $xy^{-1}$.
Suppose $t_{\gamma}$ is a mapping class. By Theorem \ref{support},
$t_{\gamma}$ is represented by a diffeomorphism whose support lies in
$N(\gamma)$. It is known that the mapping class group of a pair of
pants is the free abelian group of rank three generated
by the boundary-parallel Dehn twists. See for example, \cite{FM} \S3.6.
It follows that there exist $m_1,m_2,m_3\in \mathbb{Z}$ such that
\begin{equation}
\label{t_8mustbe}
t_{\gamma}=t_{C_1}^{m_1}t_{C_2}^{m_2}t_{C_3}^{m_3}.
\end{equation}
Let us choose a Magnus expansion $\theta$ satisfying $\theta(\zeta)=\exp(\omega)$.
Since the Dehn twists $t_{C_1}$, $t_{C_2}$, and $t_{C_3}$ commute
with each other, so do the derivations $L^{\theta}(C_1)$,
$L^{\theta}(C_2)$, and $L^{\theta}(C_3)$. Therefore,
applying $T^{\theta}$ to (\ref{t_8mustbe}) and taking the logarithm,
we obtain
\begin{equation}
\label{tobemujun}
L^{\theta}(xy^{-1})=m_1L^{\theta}(x)+m_2L^{\theta}(y)+m_3L^{\theta}(xy).
\end{equation}
In fact, for any configuration of a figure eight, we get
$$t_{\gamma}=t_{C_1}^2t_{C_2}^2t_{C_3}^{-1}$$
as an intermediate result, see Proposition \ref{22-1}.
Looking at this equation in higher degree leads us to a contradiction.
To do this, we will need the lower terms of the invariant
$L^{\theta}$. For simplicity we write $\ell=\ell^{\theta}$,
$L=L^{\theta}$ and
$$\ell(x)=\sum_{m=1}^{\infty}\ell_m(x),\ L(x)=\sum_{m=2}^{\infty} L_m(x),$$
where $\ell_m(x),L_m(x)\in H^{\otimes m}$. We have
$$L_m(x)=\frac{1}{2}\sum_{i=1}^{m-1}N(\ell_i(x)\ell_{m-i}(x)).$$
By the Baker-Campbell-Hausdorff formula, in the completed
free Lie algebra generated by variables $u$ and $v$, we have 
\begin{eqnarray}
\log (\exp u \exp v) &=&
u+v+\frac{1}{2}[u,v]+\frac{1}{12}[u-v,[u,v]]
-\frac{1}{24}[u,[v,[u,v]]] \nonumber \\
& &
+(\rm{higher\ terms}). \label{BCH}
\end{eqnarray}
For $x,y\in \pi$, we denote $X=[x], Y=[y] \in H$.
By (\ref{BCH}), we get the lower terms of $\ell(xy)=\log (\theta(x)\theta(y))$.
For example, we have
\begin{eqnarray}
\label{Haus1}
\ell_1(xy) &=& \ell_1(x)+\ell_1(y) \nonumber \\
\ell_2(xy) &=& \ell_2(x)+\ell_2(y)+\frac{1}{2}[X,Y] \nonumber \\
\ell_3(xy) &=& \ell_3(x)+\ell_3(y)+\frac{1}{2}([X,\ell_2(y)]+[\ell_2(x),Y])
+\frac{1}{12}[X-Y,[X,Y]],
\end{eqnarray}
and if $[X,Y]=0$, we have
\begin{eqnarray}
\label{Haus2}
\ell_2(xy) &=& \ell_2(x)+\ell_2(y) \nonumber \\
\ell_3(xy) &=& \ell_3(x)+\ell_3(y)+\frac{1}{2}([X,\ell_2(y)]+[\ell_2(x),Y]) \nonumber \\
\ell_4(xy) &=& \ell_4(x)+\ell_4(y)+\frac{1}{2}
([X,\ell_3(y)]+[\ell_2(x),\ell_2(y)]+[\ell_3(x),Y]) \nonumber \\
& & +\frac{1}{12}[X-Y,[X,\ell_2(y)]+[\ell_2(x),Y]]
\end{eqnarray}
and
\begin{eqnarray}
\label{Haus3}
\ell_5(xy) &=& \ell_5(x)+\ell_5(y)+\frac{1}{2}
([X,\ell_4(y)]+[\ell_2(x),\ell_3(y)]+[\ell_3(x),\ell_2(y)]+[\ell_4(x),Y]) \nonumber \\
& &+\frac{1}{12}[\ell_2(x)-\ell_2(y),[X,\ell_2(y)]+[\ell_2(x),Y]] \nonumber \\
& & +\frac{1}{12}[X-Y,[X,\ell_3(y)]+[\ell_2(x),\ell_2(y)]+[\ell_3(x),Y]] \nonumber \\
& & -\frac{1}{24}[X,[Y,[X,\ell_2(y)]+[\ell_2(x),Y]]].
\end{eqnarray}

\subsection{determination of coefficients}

The goal of this subsection is to prove Proposition \ref{22-1}.

\begin{lem}
\label{L_2(xy)}
\begin{eqnarray*}
L_2(xy) &=& L_2(x)+L_2(y)+N(XY) \\
L_2(xy^{-1}) &=& L_2(x)+L_2(y)-N(XY)
\end{eqnarray*}
\end{lem}

\begin{proof}
Since $\ell_1(x)=[x]=X$,
\begin{eqnarray*}
L_2(xy) &=& \frac{1}{2}N(\ell_1(xy)\ell_1(xy))=\frac{1}{2}N((X+Y)(X+Y)) \\
&=& \frac{1}{2}(N(XX)+N(YY)+N(XY)+N(YX)) \\
&=& L_2(x)+L_2(y)+N(XY).
\end{eqnarray*}
Here we use Lemma \ref{propN}. The other one follows from
the first one by replacing $y$ with $y^{-1}$. 
\end{proof}

\begin{lem}
\label{L_4(xy)}
Suppose $[X,Y]=0$. Then
\begin{eqnarray*}
L_4(xy) &=& L_4(x)+L_4(y)+N(X\ell_3(y)+Y\ell_3(x))
+N(\ell_2(x)\ell_2(y)) \\
L_4(xy^{-1}) &=& L_4(x)+L_4(y)-N(X\ell_3(y)+Y\ell_3(x))
-N(\ell_2(x)\ell_2(y))
\end{eqnarray*}
\end{lem}

\begin{proof}
By (\ref{Haus2}), we have
\begin{eqnarray*}
L_4(xy)&=& N(\ell_1(xy)\ell_3(xy))+\frac{1}{2}N(\ell_2(xy)\ell_2(xy)) \\
&=& N \left((X+Y)(\ell_3(x)+\ell_3(y)
+\frac{1}{2}([X,\ell_2(y)]+[\ell_2(x),Y])) \right) \\
& & +\frac{1}{2}N((\ell_2(x)+\ell_2(y))(\ell_2(x)+\ell_2(y))).
\end{eqnarray*}
By Lemma \ref{propN}, we have $N(X[X,\ell_2(y)])=N([X,X]\ell_2(y))=0$,
$N(X[\ell_2(x),Y])=-N(X[Y,\ell_2(x)])=-N([X,Y]\ell_2(x))=0$ (using
$[X,Y]=0$), etc. Therefore
$$L_4(xy)=N((X+Y)(\ell_3(x)+\ell_3(y)))
+\frac{1}{2}N((\ell_2(x)+\ell_2(y))(\ell_2(x)+\ell_2(y))).$$
Expanding the right hand side, we obtain the first formula.
The other one follows from the first one by replacing $y$ with $y^{-1}$.
\end{proof}

To advance our computation we need explicit values of a symplectic expansion.
In \cite{Mas}, Massuyeau gave some lower terms of a symplectic
expansion. If we denote it by $\theta^0$, then the values of
$\ell^{\theta^0}=\log \theta^0$ on symplectic generators are as follows: modulo $\widehat{T}_4$,
\begin{eqnarray}
\ell^{\theta^0}(\alpha_i)&\equiv & A_i+\frac{1}{2}[A_i,B_i] \nonumber \\
& & -\frac{1}{12}[B_i,[A_i,B_i]]+\frac{1}{2}\sum_{j<i}[A_i,[A_j,B_j]], \nonumber \\
\ell^{\theta^0}(\beta_i)&\equiv& B_i-\frac{1}{2}[A_i,B_i] \nonumber \\
& & +\frac{1}{12}[A_i,[A_i,B_i]]+\frac{1}{4}[B_i,[A_i,B_i]]
+\frac{1}{2}\sum_{j<i}[B_i,[A_j,B_j]] \label{masexp}.
\end{eqnarray}
Note that our conventions is different from \cite{Mas} Definition 2.15.

\begin{lem}
\label{indep1}
Suppose $h\ge 2$. Then the tensors
$u_1:=N(\textstyle{\sum}_{i=1}^h[A_i,B_i] \textstyle{\sum}_{i=1}^h[A_i,B_i])$,
$u_2:=N(\textstyle{\sum}_{i=1}^h[A_i,B_i] [A_h,B_h])$, and
$u_3:=N([A_h,B_h][A_h,B_h])$ are linearly independent.
\end{lem}

\begin{proof}
Note that the tensors $X_{i_1}\cdots X_{i_m}$, $X_{i_k} \in \{ A_i,B_i\}_i$
constitute a basis of $H^{\otimes m}$. Writing $u_i$ in terms of this basis,
we see that the coefficients of $A_1B_1A_1B_1$ in $u_1$, $u_2$, and $u_3$
are $4$, $0$, and $0$, respectively; the coefficients of $A_1B_1A_hB_h$ in
$u_1$, $u_2$, and $u_3$ are $2$, $1$, and $0$, respectively;
the coefficients of $A_hB_hA_hB_h$ in $u_1$, $u_2$, and $u_3$ are
all $4$. This shows that $u_1$, $u_2$, and $u_3$ are linearly
independent.
\end{proof}

\begin{lem}
\label{indep2}
Let $k_1,k_2 \ge 1$ and set $\omega_1:=\sum_{i=1}^{k_1} [A_i,B_i]$ and
$\omega_2:=\sum_{i=k_1+1}^{k_2}[A_i,B_i]$. Then the tensors
$N(\omega_1 \omega_1)$, $N(\omega_1 \omega_2)$, and
$N(\omega_1 \omega_2)$ are linearly independent.
\end{lem}

\begin{proof}
This is proved by the same way as in Lemma \ref{indep1}.
\end{proof}

\begin{lem}
\label{deg4indep}
Let $x$ and $y$ be one of the pair other than the case I in Table 1.
We assume $h\ge 2$ in the cases II-a and II-b.
Then for $L=L^{\theta^0}$,  the tensors
$L_4(x)$, $L_4(y)$, and
$N(X\ell_3(y)+Y\ell_3(x))+
N(\ell_2(x)\ell_2(y))$ are linearly independent.
\end{lem}

\begin{proof}
For simplicity we denote $M:=N(X\ell_3(y)+Y\ell_3(x))+N(\ell_2(x)\ell_2(y))$.
By a direct computation using (\ref{masexp}), we get the values of
$L_4(x)$, $L_4(y)$, and $M$ for $\theta=\theta^0$ as Table 2.
\begin{center}
Table 2

\begin{tabular}{l|c|c}
 & $L_4(x)$ & $L_4(y)$ \\
 \hline
 II-a & $(1/2)u_1-(1/2)u_2+(1/24)u_3$ & $(1/24)u_3$ \\
II-b & $(1/2)u_1-(1/2)u_2+(1/24)u_3$ & $(1/2)u_1$  \\
III-a & $(1/2)u_1-(1/2)u_2+(1/24)u_3$ & $(1/24)u_3$  \\
III-b & $(1/2)u_1-(1/2)u_2+(1/24)u_3$ & $(1/2)u_1-u_2-(1/2)u_3$  \\
IV-a & $(1/2)N(\omega_1 \omega_1)$ & $(1/2)N(\omega_2 \omega_2)$  \\
IV-b & $(1/2)N(\omega_1 \omega_2)$ & $(1/2)N(\omega_1 \omega_1)+
(1/2)N(\omega_2 \omega_2)+N(\omega_1 \omega_2)$ \\
\end{tabular}

\begin{tabular}{l|c}
 & $M$ \\
 \hline
 II-a & $(1/2)u_2-(1/12)u_3$ \\
II-b & $-u_1+(1/2)u_2$ \\
 III-a & $-(1/2)u_2+(5/12)u_3$ \\
 III-b & $-u_1+(3/2)u_2-(1/2)u_3$ \\
 IV-a & $N(\omega_1 \omega_2)$ \\
 IV-b & $-N(\omega_1 \omega_1)-N(\omega_1 \omega_2)$ \\
 \end{tabular}
\end{center}

These computations together with Lemmas \ref{indep1} and \ref{indep2}
give the result.
\end{proof}

Now we have an intermediate result to deduce a contradiction.

\begin{prop}
\label{22-1}
Suppose $t_{\gamma}$ is a mapping class. Then we have
$t_{\gamma}=t_{C_1}^2t_{C_2}^2t_{C_3}^{-1}$.
\end{prop}

\begin{proof}
Let $\theta^0$ be the symplectic expansion of (\ref{masexp})
and $L=L^{\theta^0}$.
If $\gamma$ is non-separating, $X=[\alpha_1]$ and $Y=[\alpha_2]$,
hence $L_2(x)=XX$, $L_2(y)=YY$, and $N(XY)$ are linearly
independent. By Lemma \ref{L_2(xy)}, the degree two part
of (\ref{tobemujun}) is equivalent to
\begin{eqnarray*}
L_2(x)+L_2(y)-N(XY) &=&
m_1L_2(x)+m_2L_2(y)+m_3(L_2(x)+L_2(y)+N(XY)) \\
&=& (m_1+m_3)L_2(x)+(m_2+m_3)L_2(y)+m_3N(XY).
\end{eqnarray*}
Comparing the coefficients, we get $(m_1,m_2,m_3)=(2,2,-1)$, i.e.,
$t_{\gamma}=t_{C_1}^2t_{C_2}^2t_{C_3}^{-1}$.
Suppose $\gamma$ is separating. If the configuration of $\gamma$ is neither
II-a nor II-b with $h=1$, the same argument applied to
the degree four part of (\ref{tobemujun}), together with
Lemmas \ref{L_4(xy)} and \ref{deg4indep} gives the result.

Suppose $h=1$ and the configuration of $\gamma$ is II-a or II-b. These
cases are rather special, since $C_1=C_2$ for the former case and
$C_1=C_3$ for the latter. If the configuration of $\gamma$ is II-a,
$x=\alpha_h\beta_h\alpha_h^{-1}$
and $y=\beta_h^{-1}$. Then $L_2(x)=L_2(y)=B_hB_h$ and $L_2(xy)=0$, $L_2(xy^{-1})=4B_hB_h$.
Looking at the degree two part of (\ref{tobemujun}), we get $m_1+m_2=4$.
Also we have $L_4(x)=L_4(y)=(1/24)u_3$, $L_4(xy)=(1/4)u_3$, and $L_4(xy^{-1})=-(1/12)u_3$.
Looking at the degree four part of (\ref{tobemujun}), we get $m_3=-1$.
Since $t_{C_1}=t_{C_2}$, we have $t_{\gamma}=t_{C_1}^2t_{C_2}^2t_{C_3}^{-1}$.
If the configuration of $\gamma$ is II-b, then $x=\alpha_h\beta_h\alpha_h^{-1}$,
$y=[\beta_h,\alpha_h]$. By a similar computation, we get $m_1+m_3=1$ and
$m_2=2$. Since $t_{C_1}=t_{C_3}$, we have $t_{\gamma}=t_{C_1}^2t_{C_2}^2t_{C_3}^{-1}$.
This completes the proof.
\end{proof}

\subsection{deduce a contradiction}

By Proposition \ref{22-1}, if $t_{\gamma}$ is a mapping class,
we have
\begin{equation}
\label{tobemujun2}
L(xy)+L(xy^{-1})=2L(x)+2L(y).
\end{equation}

Suppose $\gamma$ is non-separating. Let $\theta=\theta^0$ be the symplectic expansion
of (\ref{masexp}). By a straightforward computation,
for $x=\alpha_1$ and $y=\alpha_2$ we have
$$L_4(xy)+L_4(xy^{-1})-2L_4(x)-2L_4(y)=-\frac{1}{12}N([A_1,A_2][A_1,A_2]) \neq 0.$$
This contradicts to (\ref{tobemujun2}), which completes the proof of Theorem \ref{mainthm}
for non-separating $\gamma$.

Hereafter we assume $\gamma$ is separating.

\begin{lem}
\label{L_6(xy)}
Suppose $[X,Y]=0$. Then
\begin{equation}
\label{N[][]}
L_6(xy)+L_6(xy^{-1})-2L_6(x)-2L_6(y)
=-\frac{1}{12}N(([X,\ell_2(y)]+[\ell_2(x),Y])([X,\ell_2(y)]+[\ell_2(x),Y])).
\end{equation}
\end{lem}

\begin{proof}
We have
$$L_6(z)=N(Z\ell_5(z))+N(\ell_2(z)\ell_4(z))
+\frac{1}{2}N(\ell_3(z)\ell_3(z)).$$
We denote $L_6^{\prime}(z):=N(Z\ell_5(z))$,
$L_6^{\prime \prime}(z):=N(\ell_2(z)\ell_4(z))$,
and $L_6^{\prime \prime \prime}(z):=\frac{1}{2}N(\ell_3(z)\ell_3(z))$.
By (\ref{Haus3}),
\begin{eqnarray*}
L_6^{\prime}(xy) &=& N((X+Y)\ell_5(xy)) \\
&=& N((X+Y)(\ell_5(x)+\ell_5(y))) \\
& & +\frac{1}{2}N((X+Y)([X,\ell_4(y)]+[\ell_2(x),\ell_3(y)]
+[\ell_3(x),\ell_2(y)]+[\ell_4(x),Y])) \\
& & +\frac{1}{12}N((X+Y)[\ell_2(x)-\ell_2(y),[X,\ell_2(y)]+[\ell_2(x),Y]]) \\
& & +\frac{1}{12}N((X+Y)[X-Y,[X,\ell_3(y)]+[\ell_2(x),\ell_2(y)]+[\ell_3(x),Y]]) \\
& & -\frac{1}{24}N((X+Y)[X,[Y,[X,\ell_2(y)]+[\ell_2(x),Y]]]).
\end{eqnarray*}
By Lemma \ref{propN} and $[X,Y]=0$, the fourth term vanish:
\begin{eqnarray*}
& & N((X+Y)[X-Y,[X,\ell_3(y)]+[\ell_2(x),\ell_2(y)]+[\ell_3(x),Y]]) \\
&=& N([X+Y,X-Y]([X,\ell_3(y)]+[\ell_2(x),\ell_2(y)]+[\ell_3(x),Y]])) \\
&=& 0,
\end{eqnarray*}
so does the fifth term. Also, we have
$N((X+Y)[X,\ell_4(y)])=N([X+Y,X]\ell_4(y))=0$ and
$N((X+Y)[\ell_4(x),Y])=0$. Therefore,
\begin{eqnarray*}
L_6^{\prime}(xy) &=& N((X+Y)(\ell_5(x)+\ell_5(y))) \\
& & +\frac{1}{2}N((X+Y)([\ell_2(x),\ell_3(y)]+[\ell_3(x),\ell_2(y)])) \\
& & +\frac{1}{12}N((X+Y)[\ell_2(x)-\ell_2(y),[X,\ell_2(y)]+[\ell_2(x),Y]]).
\end{eqnarray*}
Next, by (\ref{Haus2}),
\begin{eqnarray*}
L_6^{\prime \prime}(xy) &=& N(\ell_2(xy)\ell_4(xy)) \\
&=& N((\ell_2(x)+\ell_2(y))(\ell_4(x)+\ell_4(y))) \\
& & +\frac{1}{2}N((\ell_2(x)+\ell_2(y))
([X,\ell_3(y)]+[\ell_2(x),\ell_2(y)]+[\ell_3(x),Y]) \\
& & +\frac{1}{12}N((\ell_2(x)+\ell_2(y))[X-Y,[X,\ell_2(y)]+[\ell_2(x),Y]]).
\end{eqnarray*}
By Lemma \ref{propN}, we have $N((\ell_2(x)+\ell_2(y))[\ell_2(x),\ell_2(y)])=0$.
Therefore,
\begin{eqnarray*}
L_6^{\prime \prime}(xy) &=& N((\ell_2(x)+\ell_2(y))(\ell_4(x)+\ell_4(y))) \\
& & +\frac{1}{2}N((\ell_2(x)+\ell_2(y))
([X,\ell_3(y)]+[\ell_3(x),Y]) \\
& & +\frac{1}{12}N((\ell_2(x)+\ell_2(y))[X-Y,[X,\ell_2(y)]+[\ell_2(x),Y]]).
\end{eqnarray*}
Finally, by (\ref{Haus2}),
\begin{eqnarray*}
L_6^{\prime \prime \prime}(xy) &=& \frac{1}{2}N(\ell_3(xy)\ell_3(xy)) \\
&=& \frac{1}{2}N\left(
\left(\ell_3(x)+\ell_3(y)+\frac{1}{2}([X,\ell_2(y)]+[\ell_2(x),Y]) \right)^{\otimes 2}
\right) \\
&=& \frac{1}{2}N((\ell_3(x)+\ell_3(y))(\ell_3(x)+\ell_3(y))) \\
& & +\frac{1}{2}N((\ell_3(x)+\ell_3(y))([X,\ell_2(y)]+[\ell_2(x),Y])) \\
& & +\frac{1}{8}N(([X,\ell_2(y)]+[\ell_2(x),Y])([X,\ell_2(y)]+[\ell_2(x),Y])).
\end{eqnarray*}
Summing all the three terms and using Lemma \ref{propN}, we get
\begin{eqnarray*}
L_6(xy) &=& N((X+Y)(\ell_5(x)+\ell_5(y))
+N((\ell_2(x)+\ell_2(y))(\ell_4(x)+\ell_4(y))) \\
& & +\frac{1}{2}N((\ell_3(x)+\ell_3(y))(\ell_3(x)+\ell_3(y))) \\
& & -\frac{1}{24}N(([X,\ell_2(y)]+[\ell_2(x),Y])([X,\ell_2(y)]+[\ell_2(x),Y])).
\end{eqnarray*}
Replacing $y$ with $y^{-1}$, we have
\begin{eqnarray*}
L_6(xy^{-1}) &=& N((X-Y)(\ell_5(x)-\ell_5(y))
+N((\ell_2(x)-\ell_2(y))(\ell_4(x)-\ell_4(y))) \\
& & +\frac{1}{2}N((\ell_3(x)-\ell_3(y))(\ell_3(x)-\ell_3(y))) \\
& & -\frac{1}{24}N(([X,\ell_2(y)]+[\ell_2(x),Y])([X,\ell_2(y)]+[\ell_2(x),Y])).
\end{eqnarray*}
Hence
\begin{eqnarray*}
L_6(xy)+L_6(xy^{-1}) &=&
2N(X\ell_5(x)+Y\ell_5(y))+2N(\ell_2(x)\ell_4(x)+\ell_2(y)\ell_4(y)) \\
& & +N(\ell_3(x)\ell_3(x)+\ell_3(y)\ell_3(y)) \\
& & -\frac{1}{12}N(([X,\ell_2(y)]+[\ell_2(x),Y])([X,\ell_2(y)]+[\ell_2(x),Y])).
\end{eqnarray*}
Expanding the right hand side, we get the formula.
\end{proof}

\begin{lem}
\label{N[[]][[]]}
Let $h\ge 1$. Then
$$N([B_h,\textstyle{\sum}_{i=1}^h [A_i,B_i]]
[B_h,\textstyle{\sum}_{i=1}^h [A_i,B_i]]) \neq 0.$$
If $h \ge 2$, then
$$N([B_h,\textstyle{\sum}_{i=1}^{h-1} [A_i,B_i]]
[B_h,\textstyle{\sum}_{i=1}^{h-1} [A_i,B_i]]) \neq 0.$$
\end{lem}

\begin{proof}
This is proved by the same way as in Lemma \ref{indep1}.
\end{proof}

Suppose the configuration of $\gamma$ is one of
II-a, II-b, III-a, and III-b. By a straightforward computation,
we see the right hand side of (\ref{N[][]}) in Lemma \ref{L_6(xy)}
for $L=L^{\theta^0}$ is
$$-\frac{1}{12}N([B_h,\textstyle{\sum}_{i=1}^h[A_i,B_i]]
[B_h,\textstyle{\sum}_{i=1}^h[A_i,B_i]])$$
if the configuration of $\gamma$ is II-a or II-b, and
$$-\frac{1}{12}N([B_h,\textstyle{\sum}_{i=1}^{h-1} [A_i,B_i]]
[B_h,\textstyle{\sum}_{i=1}^{h-1} [A_i,B_i]])$$
if the configuration of $\gamma$ is III-a or III-b.
By Lemma \ref{N[[]][[]]}, this contradicts to
(\ref{tobemujun2}).

Finally, we consider the cases IV-a and IV-b.

\begin{lem}
Suppose $X=Y=0$. Then
\begin{equation}
\label{L_8(xy)}
L_8(xy)+L_8(xy^{-1})-2L_8(x)-2L_8(y)=
-\frac{1}{12}N([\ell_2(x),\ell_2(y)][\ell_2(x),\ell_2(y)]).
\end{equation}
\end{lem}

\begin{proof}
This is proved by the same way as in Lemma \ref{L_6(xy)}.
We just remark that if $X=Y=0$, then by (\ref{BCH})
\begin{eqnarray*}
\ell_2(xy) &=& \ell_2(x)+\ell_2(y) \\
\ell_3(xy) &=& \ell_3(x)+\ell_3(y) \\
\ell_4(xy) &=& \ell_4(x)+\ell_4(y)+\frac{1}{2}[\ell_2(x),\ell_2(y)] \\
\ell_5(xy) &=& \ell_5(x)+\ell_5(y)+\frac{1}{2}([\ell_2(x),\ell_3(y)]
+[\ell_3(x),\ell_2(y)]) \\
\ell_6(xy) &=& \ell_6(x)+\ell_6(y)
+\frac{1}{2}([\ell_2(x),\ell_4(y)]+[\ell_3(x),\ell_3(y)]
+[\ell_4(x),\ell_2(y)]) \\
& & +\frac{1}{12}[\ell_2(x)-\ell_2(y),[\ell_2(x),\ell_2(y)]].
\end{eqnarray*}
\end{proof}

If the configuration of $\gamma$ is IV-a or IV-b,
the right hand side of (\ref{L_8(xy)}) for $L=L^{\theta^0}$ is
$$-\frac{1}{12}N([\textstyle{\sum}_{i=1}^{k_1+k_2}[A_i,B_i],
\textstyle{\sum}_{i=k_1+1}^{k_1+k_2}[A_i,B_i]]
[\textstyle{\sum}_{i=1}^{k_1+k_2}[A_i,B_i],
\textstyle{\sum}_{i=k_1+1}^{k_1+k_2}[A_i,B_i]]).$$
By the same way as in Lemma \ref{indep1}, we see that
this tensor of degree eight is not zero. This contradicts
to (\ref{tobemujun2}).

Now we have deduced a contradiction for any configuration
of a figure eight. This completes the proof of
Theorem \ref{mainthm}.

\vspace{0.3cm}

We end this paper by posing a question regarding
the characterization of simple closed curves.

\begin{quest}
{\rm Let $\gamma$ be an unoriented loop on $\Sigma$
and suppose the generalized Dehn twist $t_{\gamma}$
is a mapping class. Is $\gamma$ homotopic to a
power of a simple closed curve ?
}
\end{quest}

\noindent \textsc{Yusuke Kuno\\
Department of Mathematics,\\
Graduate School of Science,\\
Hiroshima University,\\
1-3-1 Kagamiyama, Higashi-Hiroshima, Hiroshima 739-8526 JAPAN}

\noindent \texttt{E-mail address: kunotti@hiroshima-u.ac.jp}

\end{document}